\documentclass[12pt,reqno]{amsart}
\usepackage{a4,amsmath,mathrsfs,
}
\usepackage[dvips]{graphicx}
\makeatletter \@addtoreset{equation}{section}

\makeatother \makeatletter

\newcommand{\eps}{\varepsilon}
\newcommand{\ro}{\varrho}
\newcommand{\rn}[1]{{\mathbb R}^{#1}}
\newcommand\R{{\rn{}}}
\newcommand\Z{{\mathbb Z}}

\newcommand{\segop}{[\kern-2pt[}
\newcommand{\segcl}{]\kern-2pt]}
\newcommand{\Leb}[1]{{\mathscr L}^{#1}}
\newcommand{\Haus}[1]{{\mathscr H}^{#1}}
\newcommand{\res}{\mathop{\hbox{\vrule height 7pt width .5pt depth 0pt
\vrule height .5pt width 6pt depth 0pt}}\nolimits}

\newcommand{\dist}{\mbox{dist}}
\newcommand{\modp}{{\rm mod}(p)}
\newcommand{\modtwo}{{\rm mod}(2)}
\newcommand{\fflat}{{\mathscr F}}

\newcommand{\rc}[2]{{\mathcal I}_{{#1}}({#2})}      
\newcommand{\rcp}[2]{{\mathcal I}_{p,{#1}}({#2})}   

\newcommand{\fc}[2]{{\mathscr F}_{{#1}}({{#2}})}
\newcommand{\fcstar}[2]{{\mathscr F}^*_{{#1}}({{#2}})}

\newcommand{\fcp}[2]{{\mathscr F}_{p,{#1}}({{#2}})}
\newcommand{\fcpstar}[2]{{\mathscr F}^*_{p,{#1}}({{#2}})}

\newcommand{\ic}[2]{{\bf I}_{{#1}}({#2})}       
\newcommand{\icp}[2]{{\bf I}_{p,{#1}}({#2})}    

\newcommand{\iicp}[2]{{\bf C}_{p,{#1}}({#2})}    

\newtheorem{theorem}{Theorem}[section]
\newtheorem{lemma}[theorem]{Lemma}
\newtheorem{proposition}[theorem]{Proposition}
\newtheorem{corollary}[theorem]{Corollary}

\theoremstyle{definition}
\newtheorem{definition}[theorem]{Definition}

\newtheorem{remark}[theorem]{Remark}

\numberwithin{equation}{section}
\numberwithin{figure}{section}
\numberwithin{table}{section}

\title[Flat currents and filling radius]{Flat currents modulo~$p$ in
metric spaces and filling radius inequalities}

\author{ Luigi Ambrosio}

\address{Scuola Normale Superiore, Piazza Cavalieri 7, 56100 Pisa, Italy}
\email{l.ambrosio@sns.it}

\author{ Mikhail G. Katz}

\address{Department of Mathematics, Bar Ilan University, Ramat Gan
52900, Israel}

\email{katzmik@math.biu.ac.il}



\date{\today}

\begin{document}

\maketitle

\begin{abstract}
We adapt the theory of currents in metric spaces, as developed by the
first-mentioned author in collaboration with B.~Kirchheim, to currents
with coefficients in~$\Z_p$.  Building upon S.~Wenger's work in the
orientable case, we obtain isoperimetric inequalities~$\modp$ in
Banach spaces and we apply these inequalities to provide a proof of
Gromov's filling radius inequality (and therefore also the systolic
inequality) which applies to nonorientable manifolds, as well. With
this goal in mind, we use the Ekeland principle to provide
quasi-minimizers of the mass $\modp$ in the homology class, and use
the isoperimetric inequality to give lower bounds on the growth of
their mass in balls.
\end{abstract}

\tableofcontents


Our aim is the extension of the theory of rectifiable currents in
metric and infinite-dimensional Banach spaces to the case of
coefficients in~$\Z_p$.  Such an extension can be applied to give
transparent proofs of Gromov's filling radius and filling volume
inequalities which apply to nonorientable manifolds, as well.


\section{Current history}

Following the classical paper by H. Federer and
W.~Fleming~\cite{fedfle}, as well as Federer's treatise \cite{federer}
on the theory of currents, in the last few years the theory has
undergone two important developments:

\begin{itemize}
\item[{--}] B. White's theory \cite{white}, inspired by Fleming's
paper \cite{Fl}, of rectifiable flat chains with coefficients in a
general group, in Euclidean spaces;
\item[{--}] the theory developed by the first author and B. Kirchheim
in \cite{ak2}, and inspired by E. De Giorgi~\cite{dg1}, of real and
integer rectifiable currents in general metric spaces.
\end{itemize}

A unified picture (general coefficients in general spaces) seemed to
be still missing, but after the completion of this paper we knew about
the paper by T.De Pauw and R.Hardt \cite{depauw2} and the earlier
paper by T.Adams \cite{adams}, developed in the same spirit of the
Fleming-White theory (but with no discussion of isoperimetric
inequalities).  Another valuable contribution to the literature came
even more recently with S.~Wenger's papers \cite{We}, \cite{We1} on
the isoperimetric inequalities.  The classical approach \cite{fedfle}
to proving these inequalities in arbitrary dimension and codimension
goes back to the deformation theorem.  A different technique was
introduced by M.~Gromov~\cite{Gr1} and fully exploited in \cite{ak2}.
It is based on the fact that, in finite-dimensional spaces, one can
prove isoperimetric inequalities independent not only of the
codimension, but also of the norm in the space.  Such a technique
allows one to prove the inequality in suitable metric spaces and in
infinite-dimensional spaces, provided a finite-dimensional
approximation scheme exists.

Wenger \cite{We} introduced a new ``global'' technique, based on
covering arguments and independent of deformation theorems and
finite-dimensional schemes.  His technique allows one to treat also
the case of Banach spaces to which the results in~\cite{ak2} do not
apply.  White's isoperimetric inequality \cite{white2} applies to
chains in finite-dimensional Banach spaces with coefficients in
general groups.  However, White's inequality is based on the
deformation theorem in the corresponding Euclidean space, and
therefore does not provide universal constants depending only on the
dimension of the chain.

In the present text, we follow the approach of \cite{federer} (see
also W.~Ziemer \cite{Z} for the case~$p=2$, still in Euclidean space)
to achieve an extension of the metric theory of \cite{ak2} to currents
with coefficients in~$\Z_p$: the initial idea is simply to identify
currents which differ by~$pT$, with~$T$ integer rectifiable.  But
then, since we want this equivalence to be stable under the action of
the boundary operator, it turns out that larger equivalence classes
and a suitable topology (induced by the so-called flat distances) are
needed. In any case, our currents arise as quotient classes~$[T]$ of
currents~$T$ akin to those considered in \cite{ak2}, which extend to
general spaces those of the Federer-Fleming theory.

\newcommand{\class}{\left[ \segop M \segcl_{\phantom{I\!\!\!\!}}
^{\phantom{I\!\!\!\!}} \right]}

In the simplest case~$p=2$, it is well-known that one can use currents
modulo~$2$ to describe possibly nonorientable manifolds. In
particular, we will prove in Theorem~\ref{ttanaka} that to any
compact~$n$-dimensional Riemannian manifold without boundary~$M$ one
can associate a canonical equivalence class
\[
\class
\]
(notice that the current~$\segop M\segcl$ itself is by no means
canonical) whose boundary is zero, still~$\modtwo$.  In particular,
after embedding~$M$ in a linear space, we can consider chains whose
boundary~$\modtwo$ coincides with the image of~$\segop M\segcl$.

\section{Gromov's inequalities}

A quarter century ago, M.~Gromov \cite{Gr1} initiated the modern
period in systolic geometry by proving a curvature-free~$1$-systolic
lower bound for the total volume of an essential Riemannian
manifold~$M$ of dimension~$n$.  Recall that the~$1$-systole, denoted
``Sys'', of a space is the least length of a loop that cannot be
contracted to a point in the space.  Here the term ``curvature-free''
refers to a bound independent of curvature invariants, with a constant
depending on the dimension of~$M$ (and possibly on the topology
of~$M$), but not on its geometry.  Such a bound is given by the
inequality between the leftmost and the rightmost terms in \eqref{21}
%
%
below, and can be thought of as a far-reaching generalisation of
Loewner's classical torus inequality
\begin{equation}
\label{21b}
{\rm Sys}^2 \leq \frac{2}{\sqrt{3}} {\rm Area},
\end{equation}
satisfied by every metric on the~$2$-torus, cf. \cite{Pu}.  It is
conjectured that the bound \eqref{21b} is satisfied by every surface
of negative Euler characteristic, see \cite{KK} for a detailed
discussion.  Recent publications in systolic geometry include
\cite{e7, Be08, Bru, Bru2, Bru3, EL, KK, Ka4, KR2, Sa08, HKK, KK2,
Gu09}.

The main ingredient in the proof of the inequality is Gromov's filling
inequality.  There is a certain amount of confusion in the literature
as to what constitutes Gromov's ``filling inequality''.  Gromov
actually proved several inequalities:

\begin{enumerate}
\item[{--}]
an inequality relating the filling radius and the volume.  It is
this inequality that's immediately relevant to Gromov's systolic
inequality;
\item[{--}] the inequality between the filling volume
(an~$(n+1)$-dimensional invariant) and the volume ($n$-dimensional
invariant) of~$M$.  Such an inequality can be more appropriately
referred to as an isoperimetric inequality.
\end{enumerate}

Marcel Berger performed a great deal of propaganda for systolic
geometry (see most recently \cite{Be6, Be08}).  The success of the
field is certainly due to his efforts.  In one of his popularisation
talks, he presented the following string of three inequalities:
\begin{equation}
\label{21}
{\rm Sys} \leq 6\; {\rm Fillrad} \leq Const \cdot \; {\rm
FillVol}^{1/(n+1)} \leq Const \cdot \; {\rm Vol}^{1/n}.
\end{equation}
(Here the last inequality corresponds to the isoperimetric inequality,
while the first one is sharp \cite{Ka1}.)  Berger's presentation was
intended for pedagogic purposes, but eventually led to a slight
confusion.  Namely, this string of inequalities gave the impression
that the proof breaks up into three stages, each requiring separate
treatment.  In reality, the last two inequalities are proved
simultaneously.  The technique is essentially a more precise version
of Federer-Fleming's deformation theorem.

As a matter of fact, proving the isoperimetric inequality alone does
not {\em directly\/} lead to any simplification of the proof.
Consider, for example, the familiar picture of the pseudosphere
in~$\R^3$, with a cusp along an asymptote given by the~$z$-axis.  We
think of it as a ``filling" of the unit circle in the~$(x,y)$-plane.
Alternatively, truncate the pseudosphere at large height~$z=H$, to
obtain a filling which is topologically a disk.  One immediately
realizes that the {\em filling volume\/} stays uniformly bounded, but
the {\em filling radius\/} (with respect to this particular filling)
tends to infinity.

Gromov's original proof starts by imbedding the manifold~$M$ into the
space~$L^\infty(M)$ of bounded Borel functions on~$M$.  Here a point
$x\in M$ is sent to the function~$f_x$ defined by
\begin{equation}
f_x(y) = \dist(x,y),
\end{equation}
where ``$\dist$'' is the Riemannian distance function in~$M$.  The
fact that the space~$L^\infty(M)$ is infinite-dimensional has given
some readers the impression that infinite-dimensionality of the
imbedding is an essential aspect of Gromov's proof of the systolic
inequality.  In fact, this is not the case. Indeed, we can choose a
maximal~$\epsilon$-net~$N \subset M$ with~$|N| < \infty$ points.  We
choose~$\epsilon$ satisfying~$\epsilon < \frac{1}{10} {\rm sys} (M)$.
This results in an imbedding
\begin{equation}
\label{23}
M \to \ell^\infty(N)
\end{equation}
where the systole goes down by a factor at most~$5$, see
\cite[p.~97]{SGT}.  Thus the systolic problem can easily be reduced to
finite-dimensional imbeddings.  Similarly, by choosing a sufficiently
fine~$\epsilon$-net, one can force the map \eqref{23} to be
$(1+\epsilon)$--bi-Lipschitz, for all~$\epsilon>0$ (see \cite{KK2} and
Proposition~\ref{NSA} below).  Hence finite-dimensional approximations
work well for our filling radius, as well, provided the estimates one
proves are independent of~$N$.

Gromov's original proof is difficult (a recent generalisation is
provided by L.~Guth in~\cite{Gu}; see also \cite{Gu09} and \cite{KW}).
Only the experts possess a complete understanding of the proof.  It
would thus be desirable to write down a detailed proof of Gromov's
influential theorem, and to sort out some of the confusion in the
literature.

\section{Summary of main results}

In Section~\ref{four}, we introduce flat currents and flat currents
modulo~$p$, following the traditional procedure in \cite{Z},
\cite{federer}.  The only difference is that the initial objects we
complete with respect to the flat topology are the currents
of~\cite{ak2}, whose main properties are recalled in the appendix.
Then, we see that in this class a slice operator
\[
[T]\mapsto \langle [T],u,r\rangle
\]
and a boundary operator~$[T]\mapsto\partial [T]$ are well defined.
This allows us to state a list of properties that a suitable classe of
currents, together with a suitable notion of mass, should satisfy, as
in \cite{We1}, in order to obtain the isoperimetric inequality.  The
idea is to start from the 1-dimensional isoperimetric inequality,
which needs to be directly checked, and then make a bootstrap argument
based on a clever covering argument.  Actually, as in \cite{We}, we
use the covering argument even to estabilish the 1-dimensional
isoperimetric inequality (trivial in the case of Lipschitz images of
1-dimensional simplexes considered in \cite{We1}, but not trivial in
our case). Then, we show in Section~5 and Section~6 that our class of
currents, together with a suitable notion of~$p$-mass, denoted
by~${\bf M}_p$, do satisfy the list of properties, so that an
isoperimetric inequality holds in this class.

\begin{definition}
The filling radius
\[
r([L],M)
\]
of a~$n$-dimensional cycle~$\modtwo$ in a space~$M$ is the infimum of
the numbers~$r>0$ such that, for all Banach spaces~$F$ and all
isometric embedding~$i$ of~$M$ into~$F$ there exists a~$(n+1)$
current~$[T]$~$\modtwo$ in~$F$ such that~$\partial [T]=i_\sharp [L]$
and the support of~$[T]$ is contained in the~$r$-neighbourhood of the
support of~$i_\sharp [L]$.
\end{definition}

Of course this definition makes sense only specifying the cycles we
are dealing with: they are equivalence classes~$\modtwo$ of currents
$L\in\rc{n}{E}$ whose boundary is zero, still~$\modtwo$.
Analogously, the admissible fillings~$T$ are equivalence classes
$\modtwo$ of currents in~$\rc{n+1}{E}$ whose boundary is equivalent
$\modtwo$ to~$L$ (see Section~\ref{sretti} for a precise definition
of the additive group~$\rc{n}{E}$ of integer rectifiable
$n$-currents in~$E$).

One of the main result of our paper, achieved as a
particular case of our Theorem~\ref{tkatz1} below, is the universal
upper bound
\[
r([L],M)\leq c(n)\bigl[{\bf M}_2([L])\bigr]^{1/n}.
\]
When~$M$ is a compact Riemannian manifold without boundary, applying
this result to the canonical~$n$-cycle~$[L]=\class$ in~$M$ and
setting
\begin{equation}
\label{41}
r(M)=r(M,\class)
\end{equation}
we obtain the following result.
\begin{theorem}\label{tmain}
For any compact~$n$-dimensional Riemannian manifold without boundary
the universal upper bound~$r(M)\leq c(n)\bigl[{\rm
Vol}(M)\bigr]^{1/n}$ holds.
\end{theorem}

\begin{remark}
\label{33}
Up until the proof of the isoperimetric inequalities no completeness
of our spaces of currents is really needed (closure under the action
of the slicing operator suffices). However, the proof of the universal
upper bound seems really to require some form of completeness, and
justifies the whole mathematical apparatus developed in this paper
(however, we left out many mathematical questions concerning currents
with coefficients in~$\Z_p$ that we plan to investigate in the
forthcoming paper \cite{ambkatzwen}). In order to prove our result we
use as in \cite{ak2} the Ekeland principle (valid in complete metric
spaces, see Section~\ref{sekeland} for a precise statement) to find
``quasi-minimizers'' of the~${\bf M}_p$-mass in the homology class
\[
\{[T]:\ \partial [T]=i_\sharp [L]\}
\]
and prove, using the isoperimetric inequality, that any such minimizer
has support close to the support of~$i_\sharp [L]$. Notice also that
the same argument, based on the isoperimetric inequalities, applies to
orientable manifolds: in this case the filling radius invariant
(possibly a larger one) could also be defined using the currents
in~\cite{ak2} and no quotient~$\modp$ is needed.
\end{remark}

\section{Filling radius and systole}

The invariant defined in~\eqref{41} is related to the systole by means
of the following inequality of Gromov's \cite{Gr1}, which turns out to
be sharp \cite{Ka1}.  Recall that a closed manifold~$M$ is called {\em
essential\/} if it admits a continuous map an Eilenberg-MacLane space
$K(\pi,1)$ such that the induced homomorphism in top-dimensional
homology sends the fundamental homology class of~$M$ to a nonzero
class.

\begin{theorem}[M.~Gromov]
\label{16}
Every essential~$M$ satisfies~$r(M) \geq \frac{1}{6}{\rm Sys}(M)$.
\end{theorem}

\begin{proof}
The idea of Gromov's proof is to build a retraction
skeleton-by-skeleton.  We will outline the essential idea of the
argument first, so as not to overburden the presentation with
technical details, which will be explained later.

By a {\em strongly isometric imbedding\/} we mean an imbedding of
metric spaces~$M\to V$ such that the instrinsic distance in~$M$
coincides with the ambient distance in~$V$ among points of~$M$.

%
%
We can assume without loss of generality that a piecewise linear
strongly isometric (up to epsilon) imbedding~$M \to \ell^\infty$
satisfies~$\dim(\ell^\infty) < \infty$ (see Remark~\ref{43} and
Proposition~\ref{51}).  If~$6r(M) < {\rm Sys}$, we set
\begin{equation}
\label{41b}
\epsilon = \frac{1}{10}( {\rm Sys} - 6r(M)).
\end{equation}
Consider a triangulation, extending that of (the image of)~$M$,
of~$\ell^\infty$ so each simplex has diameter at most~$\epsilon$.
If~$C$ is a current with support in the neighborhood~$U_r M$,
let~$C_{{\rm fat}}$ be the union of all simplices meeting the support
of~$C$.  Then~$C_{{\rm fat}}$ lies in the~$(r+\epsilon)$-neighborhood
of~$M$.  Let
\[
C_{{\rm fat}}^{(k)} \subset C_{{\rm fat}}
\]
be its~$k$-skeleton.  A map
\[
f^{(0)}: C_{{\rm fat}}^{(0)} \to M
\]
on the $0$-skeleton is constructed by sending each vertex to a nearest
point of~$M$.  Next, we extend~$f^{(0)}$ to a map
\[
f^{(1)}: C_{{\rm fat}}^{(1)} \to M
\]
by sending each edge to a shortest path joining the images of its
endpoints under~$f^{(0)}$, in such a way that~$f^{(1)}$ is the
identity on each edge contained in~$M$ itself (here we are assuming
that the edges of the triangulation of~$M$ are minimizing paths).
Since the distances in~$M$ coincide with the ambient distances in
$\ell^\infty$, each edge of~$C_{{\rm fat}}^{(1)}$ is mapped to a path
of length at most~$(r+\epsilon) + \epsilon + (r+ \epsilon) = 2r +
3\epsilon$.  Next, given a~$2$-simplex~$abc$ in~$C_{{\rm fat}}^{(2)}$,
note that its boundary is mapped to a loop~$L_{abc}$ of length at most
\[
3(2r + 3\epsilon) = 6r + 9\epsilon < {\rm Sys},
\]
by~\eqref{41b}, and hence~$L_{abc}$ is contractible by definition of
the systole.  We can therefore extend~$f^{(1)}$ to a map
\[
f^{(2)} : C_{{\rm fat}}^{(2)} \to M
\]
whose restriction to the intersection~$M^{(2)}\cap C_{{\rm fat}}$ is
the identity.  Every essential manifold~$M$ (see \cite{Gr1}) by
definition admits a classifying map
\[
g: M \to B\pi
\]
to the classifying space~$B\pi = K(\pi, 1)$, such that
\begin{itemize}
\item
$\pi=\pi_1(M)$;
\item
$\pi_i(B\pi)= 0$ for~$i\geq 2$,
\item
$g_*([M])\not=0$, where~$[M]$ is the fundamental class.
\end{itemize}
Therefore the composed map
\[
g\circ f^{(2)}: C_{{\rm fat}}^{(2)} \to B\pi
\]
extends to a map
\[
h: C_{{\rm fat}} \to B\pi
\]
in such a way that~$h$ coincides with~$g$ on~$M\subset C_{{\rm
fat}}^{(2)}$ (see Lemma~\ref{42b} for a more detailed statement in the
simplicial category).  Since
\[
h_*([M])= g_*([M]) \not=0,
\]
we conclude that the neighborhood~$C_{{\rm fat}}$ cannot contain a
current filling~$M$, proving the inequality.
\end{proof}

The proof above is formulated in the category of continuous maps,
which is the most convenient one in the context of classifying spaces.
On the other hand, a simplicial approximation can easily be
constructed if one works with finite skeleta of the classifying space.
The following essential lemma is standard.

\begin{lemma}
\label{42b} Consider finite dimensional simplicial
complexes~$M,Y,Z$, where~$M\subset Y$ is a subcomplex, $\dim(Y)=N$,
and~$g: M \to Z$ is continuous and simplicial, where~$\pi_i(Z)=0$
for~$i=2,\ldots,N-1$. Then given a simplicial map~$f^{(2)}:
Y^{(2)}\to M$ which is the identity on $M^{(2)}$, the composition
$g\circ f^{(2)}$ extends to a simplicial map~$h:Y\to Z$ whose
restriction to~$M\subset Y$ satisfies~$h|_M=g$.
\end{lemma}

\begin{remark}
\label{43}
Let~$N$ be a maximal~$\epsilon$-net in~$M$, and consider the finite
dimensional imbedding~$\iota: M\to \ell^\infty(N)$ whose coordinate
functions are the distance functions~$f_p$ from points~$p\in N$.  The
imbedding is not quite strongly isometric, since~$d(p,q)= \| f_p -f_q
\|$ but the functions~$f_p$ and~$f_q$ only occur as coordinates
in~$\ell^\infty$ if~$p,q$ belong to the net.  However, chooosing
nearby points~$p_0, q_0$ of the maximal net, we obtain by the triangle
inequality
\[
d(p,q) \leq d(p_0,q_0) +2\epsilon = \| f_{p_0} - f_{q_0} \| +2\epsilon
\leq \| \iota(p) - \iota(q) \| + 4\epsilon .
\]
Thus upper bounds on distances in~$\ell^\infty$ entail upper bounds on
intrinsic distances in~$M$, up to arbitrarily small error.  A more
detailed discussion may be found in Proposition~\ref{NSA}.
\end{remark}

\begin{remark}[Gromov's scheme]
Gromov's scheme, outlined in Berger \cite[p.~298]{Be5}, is to fill a
manifold~$M=M^d$ in~$\ell^\infty$ by a
minimal~$(d+1)$-submanifold~$N$.  Next, $N$ contains a point~$x$ at
distance at least~$r$ from each point of $M$.  Since~$N$ is minimal,
the volume of the distance spheres from $x$ grows sufficiently fast.
Finally, the total volume of~$N$ is at least that of a ball of
radius~$r$ in~$N$, hence at least a constant times~$r^{d+1}$.
But~${\rm Vol}(M) \geq Const \; {\rm Vol}^{d/d+1}(N)$ by the
isoperimetric inequality for minimal submanifolds (with boundary)
in~$\ell^\infty$.  Combined with the inequality of Theorem~\ref{16},
this would complete the proof of Gromov's systolic inequality.

Of course, lacking a completeness result, no notion of minimal
submanifold in Banach space was available at the time, which accounts
in part for the complications in Gromov's original proof \cite{Gr1}.
In some sense, the scheme outlined by Berger is made rigorous in the
present text, where we do have completeness, cf. Remark~\ref{33}.
\end{remark}

\section{Approximation by finite-dimensional imbeddings}

%
%

\begin{proposition}
\label{51}
\label{NSA} Let~$M$ be a compact Riemannian manifold without
boundary.  For every $\eps>0$, there exists
a~$(1+\eps)$--bi-Lipschitz finite-dimensional imbedding of~$M$,
approximating its isometric imbedding in~$L^{\infty}(M)$.
\end{proposition}

\begin{proof}
For each~$n\in {\mathbb N}$, choose a maximal~$\frac{1}{n}$-separated
net
\[
{\mathcal M}_n\subset M,
\]
and imbed~$M$ in~$\ell^\infty$ by the distance functions from the
points in the net by the 1-Lipschitz map
\begin{equation}
\label{42}
\iota_n : M \to \ell^\infty({\mathcal M}_n).
\end{equation}
If there exists a real~$\eps>0$ such that the inverse of $\iota_n$ is
not~$(1-\eps)^{-1}$--Lipschitz, then there is a pair of
points~$x_n,y_n\in M$ such that the distance~$d(x_n, y_n)$ satisfies
\begin{equation}
\label{911} | \iota_n(x_n) - \iota_n(y_n) | \leq (1-\eps) d(x_n,
y_n),
\end{equation}
meaning
\begin{equation}
\label{43b} | d(x_n, z) - d(y_n, z) | \leq (1-\eps) d(x_n, y_n)
\qquad\forall z\in {\mathcal M}_n.
\end{equation}
Since $M$ is compact, we can assume with no loss of generality that
$x_n\to x$ and $y_n\to y$, and if $x\neq y$ we can contradict
\eqref{43b} by choosing $z_n\in M_n$ at distance less than $1/n$
from $x$ and $n$ large enough. So, $x=y$ and we denote
$s_n=d(x_n,y_n)\to 0$.

Since $M$ is compact and locally bi-Lipchitz to an Euclidean space
(with Lipschitz constant close to $1$ provided we choose sufficiently
small neighbourhoods), for any $\delta>0$ we can find $\bar b>0$ such
that all (geodesic) triangles in $M$ with side lengths less than $\bar
b$ have sum of the internal angles less $2\pi+\delta$; we choose
$\delta$ in such a way that $1-\eps/2<\cos\delta$ and we assume with
no loss of generality that $\bar b\leq{\rm InjRad}(M)$.

Let $v_n\in T_{x_n}M$ be the unit vector such that $y_n={\rm
exp}_{x_n}(s_nv_n)$, set $q_n:={\rm exp}_{x_n}(\tfrac{1}{2}\bar b
v_n)$ and denote by $a_n \in {\mathcal M}_n$ a point of the maximal
net nearest to~$q_n$. Denoting by $\alpha_n$ be the angle at~$x_n$
of the geodesic triangle having $a_n,y_n,x_n$ as vertices
\[
\alpha_n := \angle a_n x_n y_n
\]
we have the Taylor expansion
\begin{equation}\label{exp_dist}
d(a_n,{\rm exp}_{x_n}(sv_n))=d(a_n,x_n)-s\cos\alpha_n+s\omega_n(s)
\end{equation}
where, thanks to the smoothness of $d$ in both variables,
$\sup_n|\omega_n(s)|$ is infinitesimal as $s\downarrow 0$.

We claim that $\alpha_n<\delta$ for $n$ large enough; indeed, the
angle at $y_n$ in the geodesic triangle having $a_n,y_n,q_n$ as
vertices tends to $0$ because the length of the side from $q_n$ to
$a_n$ tends to $0$, while the length of the other two sides does
not. As a consequence the angle at $y_n$ in the geodesic triangle
having $a_n,y_n,x_n$ as vertices tends to $\pi$. Since all sides of
the latter triangle are shorter than $\bar b$ for $n$ large enough,
our choice of $\bar b$ ensures that the angle $\alpha_n$ is less
than $\delta$ for $n$ large enough. Putting $s=s_n$ in
\eqref{exp_dist} we get
$$
|d(a_n,y_n)-d(a_n,x_n)|=s_n\cos\alpha_n+o(s_n)>
(1-\eps/2)s_n+s_n\omega_n(s_n)=(1-\eps/2)s_n+o(s_n)
$$
contradicting \eqref{43b} for $n$ large.
\end{proof}

\section{Preliminary definitions}
\label{four}

Let~$(E,d_E)$ be a metric space and~$k\geq 0$ integer. We assume,
since this suffices for our purposes, that~$(E,d_E)$ is separable;
this assumption is needed to avoid subtle measurability problems
(assuming that the cardinality of $E$ is an Ulam number this
assumption could be avoided, see \cite[2.1.6]{federer} and
\cite[Lemma~2.9]{ak2}). We use the standard notation $B_r(x)$ for the
open balls in $E$, ${\rm Lip}(E)$ for the space of Lipschitz
real-valued functions, relative to $d_E$, and ${\rm Lip}_b(E)$ for
bounded Lipschitz functions.

We consider, as in \cite{ak2}, the space $MF_k(E)$ of
$k$-dimensional currents in $E$. We denote by ${\bf M}(T)$ the mass
of $T\in MF_k(E)$, possibly infinite. We recall the basic
definitions of mass, support, push-forward, restriction, boundary in
the appendix.

Spaces of currents in $E$ are defined as in \cite{ak2}, with the
same notation, we will only use~$\rc{k}{E}$ (integer rectifiable
currents with finite mass) and $\ic{k}{E}$ (currents in $\rc{k}{E}$
whose boundary belongs to $\rc{k-1}{E}$), see Section~\ref{sretti}.
In the sequel $p\geq 2$ is a given integer.

\subsection{Flat integer currents}
We shall denote by $\fc{k}{E}$ the currents in $MF_k(E)$ that can by
written as $R+\partial S$ with $R\in\rc{k}{E}$ and
$S\in\rc{k+1}{E}$. It is obviously an additive (Abelian) group and
\begin{equation}\label{amb2}
T\in\fc{k}{E}\qquad\Longrightarrow\qquad\partial T\in\fc{k-1}{E}.
\end{equation}
$\fc{k}{E}$ is a metric space when endowed with the the distance
$d(T_1,T_2)=\fflat(T_1-T_2)$, where
$$
\fflat(T):=\inf\left\{{\bf M}(R)+{\bf M}(S):\ R\in\rc{k}{E},\,\,
S\in\rc{k+1}{E},\,\,T=R+\partial S\right\}.
$$
The subadditivity of $\fflat$, namely $\fflat(nT)\leq n\fflat(T)$,
ensures that $d$ is a distance, and the completeness of the groups
$\rc{k}{E}$, when endowed with the mass norm, ensures that
$\fc{k}{E}$ is complete. Also, whenever $\ic{k}{E}$ is dense in
$\rc{k}{E}$ (see Proposition~\ref{plindenstrauss} for sufficient
conditions), the subset
$$
\left\{R+\partial S:\
R\in\ic{k}{E},\,\,S\in\ic{k+1}{E}\right\}\subset\ic{k}{E}
$$
is dense in $\fc{k}{E}$. For the special class of currents $T$ in
$\fc{k}{E}$ with finite mass the density result can be strengthened:
indeed, if $T=T_i+R_i+\partial S_i$ with $T_i\in\ic{k}{E}$,
$R_i\in\rc{k}{E}$, $S_i\in\rc{k+1}{E}$ and ${\bf M}(R_i)+{\bf
M}(S_i)\to 0$, then Theorem~\ref{tbrett} gives $S_i\in\ic{k+1}{E}$
(because $\partial S_i$ has finite mass) hence $T_i+\partial
S_i\in\ic{k}{E}$. So, $T$ can be approximated in the stronger mass
norm by the currents $T_i+\partial S_i$ and this yields
\begin{equation}\label{comedia}
\left\{T\in\fc{k}{E}:\ {\bf M}(T)<\infty\right\}=\rc{k}{E}.
\end{equation}

Notice also that
\begin{equation}\label{amb1}
\fflat(\partial T)\leq\fflat(T),\qquad\qquad \forall T\in\fc{k}{E}.
\end{equation}
In addition, since $\partial (\varphi_\sharp
S)=\varphi_\sharp(\partial S)$ we get
\begin{equation}\label{amb3bis}
\fflat(\varphi_\sharp T)\leq [{\rm Lip}(\varphi)]^k\fflat(T)
\end{equation}
for all $T\in\fc{k}{E}$, $\varphi\in {\rm Lip}(E,\R^k)$.

It should also be emphasized that the concepts introduced in this
section are sensitive to the ambient space, namely if $E$ embeds
isometrically in $F$ then, for $T\in\fc{k}{E}$,
 $\fflat(i_\sharp T)$ can well be strictly smaller than $\fflat (T)$;
 the same remark applies
to the ${\bf M}_p$ mass, built in Section~\ref{sdefmp}.
This is not the case for the concepts of mass, a genuine
isometric invariant, see \cite{ak2}.

\subsection{Flat distance modulo $p$}
For $T\in\fc{k}{E}$ we define:
$$
\fflat_p(T):=\inf\left\{\fflat(T-pQ):\ Q\in\fc{k}{E}\right\}.
$$
The definition of $\fflat$ gives
$$
\fflat_p(T)=\inf\left\{{\bf M}(R)+{\bf M}(S):\ T=R+\partial
S+pQ,\,\,R\in\rc{k}{E},\,\,
S\in\rc{k+1}{E},\,\,Q\in\fc{k}{E}\right\}.
$$
Furthermore, whenever $\ic{k}{E}$ is dense $\fc{k}{E}$, both infima
are unchanged if $Q$ runs in $\ic{k}{E}$.

Obviously  $\fflat_p(T)\leq\fflat(T)$, and \eqref{amb1} together
with \eqref{amb2} give
\begin{equation}\label{amb3}
\fflat_p(\partial T)\leq\fflat_p(T),\qquad\quad
 T\in \fc{k}{E},
\end{equation}
while \eqref{amb3bis} gives
\begin{equation}\label{amb3ter}
\fflat_p(\varphi_\sharp T)\leq [{\rm Lip}(\varphi)]^k\fflat_p(T)
\end{equation}
for all $T\in\fc{k}{E}$, $\varphi\in {\rm Lip}(E,\R^k)$.

We now introduce an equivalence relation $\modp$ in $\fc{k}{E}$,
compatible with the group structure, by saying that $T=\tilde{T}$
$\modp$ if $\fflat_p(T-\tilde{T})=0$, and denote by $\fcp{k}{E}$ the
quotient group. Clearly $T=0$ $\modp$ if $T=pQ$ for some
$Q\in\fc{k}{E}$, but the converse implication is not known, not even
in Euclidean spaces.

The equivalence classes are closed in $\fc{k}{E}$ and by
\eqref{amb3} the boundary operator can be defined also in the
quotient spaces $\fcp{k}{E}$ in such a way that
$$
\partial [T]=[\partial T]\in\fcp{k-1}{E}\qquad\forall T\in\fc{k}{E}.
$$
The same holds, thanks to \eqref{amb3ter}, for the push-forward
operator, defined in such a way to commute with the equivalence
relation $\modp$. We emphasize that $\fcp{k}{E}$, when endowed with
the distance induced by $\fflat_p$, is a complete metric space: to
see this, let $([T_h])\subset\fcp{k}{E}$ be a Cauchy sequence and
assume with no loss of generality that
\[
\sum_h\fflat_p(T_{h+1}-T_h)<\infty;
\]
we can find $R_h\in\rc{k}{E}$, $S_h\in\rc{k+1}{E}$ and
$Q_h\in\fc{k}{E}$ such that
\[
T_{h+1}=T_h+R_h+\partial S_h+pQ_h\qquad\text{and}\qquad
\sum_{h=1}^\infty{\bf M}(R_h)+{\bf M}(S_h)<\infty.
\]
Setting $\tilde{T}_h:=T_h-p\sum_0^{h-1}Q_h$ it follows that
$\tilde{T}_h=T_h$ $\modp$ and since
$\tilde{T}_{h+1}-\tilde{T}_h=R_h+\partial S_h$ it follows that
$(\tilde{T}_h)$ is a Cauchy sequence in $\fc{k}{E}$. Denoting by $T$
its limit, by the inequality $\fflat_p\leq\fflat$ we infer
$[T_h]=[\tilde{T}_h]\to [T]$ in $\fcp{k}{E}$.

\section{Restriction, slicing}

The restriction and slicing operators can be easily extended to the
set $\fcstar{k}{E}$, defined as the closure in $\fc{k}{E}$ of
currents in $\ic{k}{E}$, using a completion argument. In the cases
considered in Proposition~\ref{plindenstrauss}, this closure
coincides with the whole of $\fc{k}{E}$ and, in any case, it is
easily seen that $\partial$ maps $\fcstar{k}{E}$ into
$\fcstar{k-1}{E}$.

Recall from \cite{ak2} that, for $u\in {\rm Lip}(E)$ and $T$ having
finite mass and boundary of finite mass the slice operator $\langle
T,u,r\rangle\in MF_{k-1}(E)$ is defined by
$$
\langle T,u,r\rangle:=\partial (T\res\{u<r\})-(\partial
T)\res\{u<r\}.
$$
Notice that $\partial\langle T,u,r\rangle=-\langle\partial
T,u,r\rangle$. It turns out that for $\Leb{1}$-a.e. $r\in\R$
$\langle T,u,r\rangle$ has finite mass, and
\begin{equation}\label{localization}
{\bf M}(\langle T,u,r\rangle)\leq {\rm
Lip}(u)\frac{d}{dr}\|T\|(\{u<r\}).
\end{equation}
Now, let $T$ be with finite mass; since $T=R+\partial S$ with
$R\in\rc{k}{E}$ and $S\in\rc{k+1}{E}$ imply that $\partial S$ has
finite mass we can apply the slicing operator to $S$ to obtain
$$
T\res\{u<r\}=R\res\{u<r\}+(\partial S)\res\{u<r\}= R\res\{u<r\}+
\partial (S\res\{u<r\})-\langle S,u,r\rangle.
$$
Since $\langle S,u,r\rangle$ belongs to $\rc{k}{E}$
for $\Leb{1}$-a.e. $r\in\R$, thanks to Proposition~\ref{pslicei},
by integration between $m$ and $\ell$ we obtain
\begin{eqnarray*}
\int_m^{*\ell}\fflat(T\res\{u<r\})\,dr&\leq&\int_m^\ell{\bf
M}(R\res\{u<r\})+{\bf M}(S\res\{u<r\})\,dr+{\rm Lip}(u)\|S\|(\{u<\ell\})\nonumber\\
&\leq& (\ell-m){\bf M}(R)+(\ell-m+{\rm Lip}(u)){\bf M}(S)
\end{eqnarray*}
where $\int^*$ denoted the upper integral (we use it to avoid the
discussion of the measurability of the map
$r\mapsto\fflat(T\res\{u<r\})$). Since $R$ and $S$ are arbitrary we
get
\begin{equation}\label{cachan}
\int_m^{*\ell}\fflat(T\res\{u<r\})\,dr\leq (\ell-m+{\rm
Lip}(u))\fflat(T).
\end{equation}
 Now, let $T\in\fc{k}{E}$, assume that there exist $T_n\in\fc{k}{E}$
 with finite mass convergent to $T$ in $\fc{k}{E}$ (this surely holds if
 $T\in\fcstar{k}{E}$), with
 $\sum_n\fflat(T_n-T)<\infty$,
and let $u\in {\rm Lip}(E)$. By adding the inequalities
\eqref{cachan} relative to $T_{n+1}-T_n$, and taking into account
the subadditivity of the outer integral and the fact that $\ell$ and
$m$ are arbitrary, we obtain that $(T_{n+1}\res\{u<r\})$ is a Cauchy
sequence in $\fc{k}{E}$ for $\Leb{1}$-a.e. $r\in\R$.

It follows that for any such $T$ we can define
\begin{equation}\label{assemblea}
T\res\{u<r\}:=\lim_{n\to\infty} T_n\res\{u<r\}\in\fc{k}{E}
\end{equation}
whenever the limit exists. By construction the operator $T\mapsto
T\res\{u<r\}$ is additive and \eqref{cachan} still holds when
$T\in\fc{k}{E}$. A similar argument shows that this definition is
independent, up to Lebesgue negligible sets, on the chosen
approximating sequence $(T_n)$, provided the ``fast convergence''
condition $\sum_n\fflat(T_n-T)<\infty$ holds.

Having defined the restriction, the slice operator, mapping currents
in $\fcstar{k}{E}$ into currents in $\fcstar{k-1}{E}$, can be again
defined by
$$
\langle T,u,r\rangle:= \partial (T\res\{u<r\})-(\partial
T)\res\{u<r\})
$$
whenever the right hand side is defined. We still have the property
$\partial \langle T,u,r\rangle=-\langle\partial T,u,r\rangle$.

From \eqref{cachan} we immediately get
\begin{equation}\label{cachanbis}
\int_m^{*\ell}\fflat_p(T\res\{u<r\})\,dr\leq (\ell-m+{\rm
Lip}(u))\fflat_p(T).
\end{equation}
In particular  $\fflat_p(T)=0$ implies $\fflat_p(T\res\{u<r\})=0$
for $\Leb{1}$-a.e. $r\in\R$, so that the restriction operator can
also be viewed as an operator in the quotient spaces
$$\fcpstar{k}{E}:=\{[T]:\ T\in\fcstar{k}{E}\},$$
with the property
$$
[T]\res\{u<r\}=[T\res\{u<r\}]\qquad\text{for $\Leb{1}$-a.e.
$r\in\R$.}
$$
Accordingly, the same holds for the slice operator, satisfying
$\partial \langle [T],u,r\rangle=-\langle\partial [T],u,r\rangle$
and
$$
\langle [T],u,r\rangle=[\langle T,u,r\rangle] \qquad\text{for
$\Leb{1}$-a.e. $r\in\R$.}
$$

\section{Isoperimetric inequalities}\label{swenger}

In this section we discuss the validity of isoperimetric
inequalities $\modp$ in suitable subspaces
$\iicp{k}{E}\subset\fcpstar{k}{E}$ analogous to those valid in the
case of currents with integer coefficients. We follow, as in
\cite{We1}, an axiomatic approach: we assume the existence, given
these subspaces  $\iicp{k}{E}$, of a notion of $p$-mass ${\bf
M}_p:\iicp{k}{E}\to\R$ satisfying the following property:

\begin{definition}[Additivity]
For all $[T]\in\iicp{k}{E}$ there exists a $\sigma$-additive Borel
measure $\|T\|_p$ satisfying
$$
{\bf M}_p([T]\res\{u<r\})=\|T\|_p(\{u<r\})\qquad\text{for
$\Leb{1}$-a.e. $r\in\R$}
$$
for all $u\in {\rm Lip}(E)$.
\end{definition}

Strictly speaking, we should use the notation $\|[T]\|$ to emphasize
that the measure depends only on the equivalence class of $T$, but
we opted for a simpler notation.

Then, we assume that $\iicp{k}{E}$ and ${\bf M}_p$ are well-behaved
with respect to the slice operator, and satisfy the isoperimetric
inequality for $1$-dimensional currents and the homogeneous version
of the isoperimetric inequality (typically achieved by a simple cone
construction):

\begin{itemize}
\item[(i)] For $k\geq 1$ the slice operator $\langle [T],u,r\rangle$ maps $\iicp{k}{E}$
into $\iicp{k-1}{E}$ and
\begin{equation}\label{coarea}
{\rm Lip}(u)\frac{d}{dr}{\bf M}_p([T]\res\{u<r\})\geq {\bf
M}_p(\langle [T],u,r\rangle) \qquad\text{for $\Leb{1}$-a.e.
$r\in\R$.}
\end{equation}
\item[(ii)] For some constant $c^*$ the following
holds: for all $[L]\in\iicp{1}{E}$ with $\partial [L]=0$ and bounded
support there exists $[T]\in\iicp{2}{E}$ with $\partial [T]=[L]$ and
$$
{\bf M}_p([T])\leq c^* \bigl[{\bf M}_p([L])\bigr]^2.
$$
In addition, if $[L]$ is supported in a ball $B$, we may choose
$[T]$ supported in the same ball.
\item[(iii)] For some constant $c_k$ the following holds:
for all $[L]\in\iicp{k}{E}$ with $\partial [L]=0$ and support
contained in a ball with radius $R$ there exists
$[T]\in\iicp{k+1}{E}$ supported in the same ball with $\partial
[T]=[L]$ and
$$
{\bf M}_p([T])\leq c_k R {\bf M}_p([L]).
$$
\item[(iv)] For some constant $A_k>0$, the following holds:
for all $[T]\in\iicp{k}{E}$ we have
$$
\liminf_{r\downarrow 0}\frac{\|T\|_p(B_r(x))}{r^k}\geq A_k
\qquad\text{$\|T\|_p$-a.e.}
$$
\end{itemize}

Given these properties, the nice and constructive decomposition
argument in \cite{We,We1} (that we reproduce in part in
Theorem~\ref{tcheck} to prove the initial isoperimetric inequality
(ii)) provides the following result:

\begin{theorem}[Isoperimetric inequality $\modp$]\label{tisop}
Assume that $E$, $\iicp{k}{E}$ and ${\bf M}_p$ fulfil the additivity
property and conditions (i), (ii), (iii), (iv). Then, for $k\geq 1$
there exist constants $\gamma_k$ such that, if $[L]\in\iicp{k}{E}$
has bounded support and satisfies $\partial [L]=0$, there exists
$[T]\in\iicp{k+1}{E}$ with $\partial [T]=[L]$ and
$$
{\bf M}_p([T])\leq\gamma_k\bigl[{\bf M}_p([L])\bigr]^{(k+1)/k}.
$$
For $k\geq 2$ the constant $\gamma_k$ depends on $\gamma_{k-1}$,
$c_k$, $A_k$.
\end{theorem}

\begin{proof} The proof is by induction on $k\geq 1$; in order to apply
the construction of \cite{We1} one needs to assume inductively that
$[T]$ can be chosen with support in a ball $B$ whenever $[L]$ is
supported in the ball. The case $k=1$ being covered by assumption
(i) and the induction step goes exactly as in \cite{We1}.
\end{proof}

\section{Definition of ${\bf M}_p$}\label{sdefmp}

For $T\in\fc{k}{E}$, its (relaxed) mass modulo $p$ is defined by:
\begin{equation}\label{mp}
{\bf M}_p(T):=\inf\left\{\liminf_{h\to\infty}{\bf M}(T_h):
T_h\in\rc{k}{E},\,\,\fflat_p(T_h-T)\to 0\right\}
\end{equation}
with the convention ${\bf M}_p(T)=+\infty$ if no approximating
sequence $(T_h)$ with finite mass exists. If $\ic{k}{E}$ is dense in
$\rc{k}{E}$ in mass norm then, as we already observed,
$\fcstar{k}{E}=\fc{k}{E}$ and flat chains with finite mass can be
approximated in mass by currents in $\ic{k}{E}$. Therefore, under
this assumption, the infimum is unchanged is we require the
approximating currents $T_h$ to be in $\ic{k}{E}$.

Obviously ${\bf M}_p\leq {\bf M}$ and ${\bf M}_p(\tilde{T})={\bf
M}_p(T)$ if $\fflat_p(\tilde{T}-T)=0$; finally, $T\mapsto{\bf
M}_p(T)$ is lower semicontinuous with respect to
$\fflat_p$-convergence. Actually, it is easy to check that ${\bf
M}_p$ is the largest functional, among those bounded above by ${\bf
M}$, with all these properties: it follows in particular that ${\bf
M}_p(T)\geq\fflat_p(T)$. We can think of ${\bf M}_p$ also as a map
defined in the quotient groups $\fcp{k}{E}$ and we shall not use a
distinguished notation for it.

\begin{theorem}\label{teomassp}
Assume that $E$ is compact. For all $[T]\in\fcp{k}{E}$ with ${\bf
M}_p([T])<\infty$ there exists a finite, nonnegative and
$\sigma$-additive Borel measure $\|T\|_p$ such that
\begin{equation}\label{garibaldi}
{\bf M}_p([T]\res\{u<r\})=\|T\|_p(\{u<r\})\qquad\text{for
$\Leb{1}$-a.e. $r\in\R$}
\end{equation}
for all $u\in {\rm Lip}(E)$.
\end{theorem}

\noindent {\sc Proof.} Let $(T_i)\subset\rc{k}{E}$ be such that
${\bf M}(T_i)\to{\bf M}_p(T)$ and $\fflat_p(T_i-T)\to 0$. Possibly
extracting a subsequence we can assume without loss of generality
that
\[
\sum_i\fflat_p(T_i-T)<\infty
\]
and that $\|T_i\|$ weakly converge, in the duality with $C(E)$, to
some finite, nonnegative and $\sigma$-additive Borel measure
$\nu$. Obviously $\nu(E)={\bf M}_p(T)$ and we claim that $\nu$
fulfills~\eqref{garibaldi}.  Indeed, let $u\in {\rm Lip}(E)$ be fixed
and let us adopt the notation
\[
R\res\{u>r\}
\]
for $R\res\{-u<-r\}$; by \eqref{cachanbis} we infer that for
$\Leb{1}$-a.e. $r\in\R$, one has that $(T_i\res\{u<r\})$ and
$(T_i\res\{u>r\})$ are Cauchy sequences with respect to $\fflat_p$
and the sum of their limits is $T$ (indeed, since $T_h$ have finite
mass,
\[
T_h=T_h\res\{u<r\}+T_h\res\{u>r\}
\]
with at most countably many exceptions). Then, denoting by
$T\res\{u<r\}$ and $T\res\{u>r\}$ the respective limits, the lower
semicontinuity of ${\bf M}_p$ gives
$$
{\bf M}_p(T\res\{u<r\})\leq
\liminf_{i\to\infty}\|T_i\|(\{u<r\}),\qquad {\bf
M}_p(T\res\{u>r\})\leq\liminf_{i\to\infty}\|T_i\|(\{u>r\})
$$
The subadditivity of ${\bf M}_p$ yields
\begin{eqnarray*}
{\bf M}_p(T)&\leq& {\bf M}_p(T\res\{u<r\})+ {\bf
M}_p(T\res\{u>r\})\\&\leq&
\liminf_{i\to\infty}\|T_i\|(\{u<r\})+\liminf_{i\to\infty}\|T_i\|(\{u>r\})\\
&\leq&
\limsup_{i\to\infty}\|T_i\|(\{u<r\})+\liminf_{i\to\infty}\|T_i\|(\{u>r\})\\
&\leq& \limsup_{i\to\infty}\|T_i\|(E)={\bf M}_p(T).
\end{eqnarray*}
It follows that all inequalities are equalities. Hence,
\[
\|T_i\|(\{u>r\})\to{\bf M}_p(T\res\{u>r\})
\]
for $\Leb{1}$-a.e.  $r\in\R$. But, thanks to the weak convergence of
$\|T_i\|$ to $\nu$, we have also
\[
\|T_i\|(\{u>r\})\to \nu(\{u>r\})
\]
with at most countably many exceptions (corresponding to the numbers
$r$ such that $\nu(\{u=r\})>0$, see for instance
\cite[Proposition~1.62(b)]{afp}). This proves \eqref{garibaldi}.  \qed

Using the measure $\|T\|_p$ we can define the support of
$T\in\fcpstar{k}{E}$, when $T$ has finite ${\bf M}_p$ mass.

\begin{definition}[Support]
Assume that $E$ is compact and that $[T]\in\fcpstar{k}{E}$ has finite
${\bf M}_p$ mass. We denote by ${\rm supp\,}[T]$ the support of the
measure $\|T\|_p$, namely $x\in {\rm supp\,}[T]$ if and only if
$\|T\|_p(B_r(x))>0$ for all $r>0$.
\end{definition}

\section{Definitions of $\icp{k}{E}$}\label{sretti}

In this section we define classes $\icp{k}{E}$ in such a way that the
properties listed in Section~\ref{swenger} hold with
$\iicp{k}{E}=\icp{k}{E}$, so that the isoperimetric inequality holds
in~$\iicp{k}{E}$.

\subsection{Currents $\segop\theta\segcl$}

Recall that, for $\theta\in L^1(\R^k)$, $\segop\theta\segcl\in
MF_k(\R^k)$ is the $k$-current in $\R^k$ defined by
$$
\segop \theta\segcl(f_0\,d\pi_1\wedge\ldots\wedge d\pi_k)=
\int_{\R^k}\theta f_0{\rm det}\nabla\pi\,dx.
$$
The change of variables formula for Lipschitz maps immediately gives
\begin{equation}\label{area}
f_\sharp\segop\theta\segcl=\segop (\sigma\theta)\circ f^{-1}\segcl
\end{equation}
whenever $f$ is a Lipschitz and injective map from $\{f\neq
0\}\subset\R^k$ to $\R^k$. Here $\sigma(x)\in\{-1,1\}$ is the sign
of the jacobian determinant of $\nabla f(x)$ (recall that points $x$
where $\sigma(x)$ is not defined, i.e. $\nabla f(x)$ is singular,
are mapped to a Lebesgue negligible set, and so they are
irrelevant).

\subsection{Countably $\Haus{k}$-rectifiable sets and integer rectifiable currents}
Denoting by $\Haus{k}$ the Hausdorff $k$-dimensional measure in $E$,
we recall also that a set $S\subset E$ is said to be countably
$\Haus{k}$-rectifiable if we can find countably many Borel sets
$B_i\subset\R^k$ and Lipschitz maps $f_i:B_i\to E$ such that
\[
\Haus{k}(S\setminus\cup_i f_i(B_i))=0.
\]
More precisely, we can also find by an exhaustion argument compact
sets $K_i\subset\R^k$ and $f_i:K_i\to E$ Lipschitz such that
$f_i(K_i)$ are pairwise disjoint and $\Haus{k}(S\setminus\cup_i
f_i(K_i))=0$. Furthermore, possibly refining once more the
partition, one can assume that $f_i:K_i\to f_i(K_i)$ are invertible
with a Lipschitz inverse (in short, bi-Lipschitz), see
\cite[Lemma~4.1]{ak2}. In the case $k=0$ we identify countably
$\Haus{k}$-rectifiable sets with finite or countable sets.

\begin{definition}[Rectifiable and integer rectifiable currents]
We say that $T\in MF_k(E)$ with finite mass is rectifiable if
$\|T\|$ vanishes on $\Haus{k}$-negligible sets and it is
concentrated on a countably $\Haus{k}$-rectifiable set. We say that
$T$ is integer rectifiable if, in addition, for all $\varphi\in {\rm
Lip}(E,\R^k)$ and all Borel sets $A$ it holds $\varphi_\sharp (T\res
A)=\segop\theta\segcl$ for some integer valued $\theta\in
L^1(\R^k)$.
\end{definition}

In the case $k=0$ rectifiable currents are finite or countable
series of Dirac masses, with integer coefficients in the integer
case, see \cite[Theorem~4.3]{ak2}. In this latter case, finiteness
of mass implies that the sum is finite.

We shall denote by $\rc{k}{E}$ the space of integer rectifiable
currents. We shall also denote by $\ic{k}{E}$ the subspace
$$
\ic{k}{E}:=\left\{T\in\rc{k}{E}:\ \partial T\in\rc{k-1}{E}\right\}.
$$

In connection with integer rectifiable currents, let us recall the
following important result (see \cite[Theorem~8.6]{ak2}):

\begin{theorem}[Boundary rectifiability]\label{tbrett}
If $T$ is integer rectifiable and has boundary with finite mass,
then $\partial T$ is integer rectifiable.
\end{theorem}

If $E$ is a closed convex subset of a Banach space the slicing
operator makes sense in $\rc{k}{E}$, thanks to Proposition~\ref{plindenstrauss},
and it enjoys the following properties (see \cite[Theorem~5.7]{ak2}):

\begin{proposition}[Slices of integer rectifiable
currents]\label{pslicei} Let $E$ be a closed convex subset of
a Banach space, $T\in\rc{k}{E}$ and $u\in {\rm
Lip}(E)$. Then $\langle T,u,r\rangle\in\rc{k-1}{E}$ for
$\Leb{1}$-a.e. $r\in\R$ and
$$
T\res du=\int_{{\bf R}}\langle T,u,r\rangle\,dr,\qquad \|T\res
du\|=\int_{{\bf R}}\|\langle T,u,r\rangle\|\,dr.
$$
\end{proposition}

It turns out the the minimal (in $\Haus{k}$-measure) set $S$ on
which $T$ is concentrated is
\begin{equation}\label{ST}
S_T:=\left\{x\in E:\ \liminf_{r\downarrow 0}r^{-k} \|T\|(B_r(x))>0
\right\}.
\end{equation}

\subsection{Multiplicity of integer rectifiable currents and reductions $\modp$}

The \emph{multiplicity} $\theta$ of a rectifiable current $T\in
MF_k(E)$ can be defined as follows: when $E=\R^k$ the multiplicity
of $\segop\theta\segcl$ is $\theta$; in general, let us represent a
Borel set $S$ on which $\|T\|$ is concentrated (i.e.
$\|T\|(E\setminus S)=0$) as $\cup_i f_i(K_i)$ with $K_i\subset\R^k$
compact, $f_i:K_i\to f_i(K_i)$ bi-Lipschitz and $f_i(K_i)$ pairwise
disjoint. Then, denoting by $g_i:E\to\R^k$ Lipschitz maps such that
$g_i\circ f_i(x)=x$ on $K_i$, we define $\theta(y)$ at $y\in
f_i(K_i)$ as the multiplicity of $(g_i)_\sharp (T\res f_i(K_i))$ at
$g_i(y)\in K_i$. Using \eqref{area} it is not difficult to check
that this definition is well posed on $S$ up to the sign and up to
$\Haus{k}$-negligible sets, i.e. that $|\theta|$ does not depend on
the chosen partition and on the Lipschitz maps $f_i$ up to
$\Haus{k}$-negligible sets (when $E$ is a linear space see also \S9
of \cite{ak2} for a definition of multiplicity closer to the one of
the Federer-Fleming theory; since this definition uses the quite
technical concept of approximate tangent space here we avoid it).
Notice also that we allow, for simplicity, the multiplicity to
vanish: but the multiplicity is nonzero $\Haus{k}$-a.e. on the set
$S_T$.

If $m\in\Z$ we call reduction of $m$ $\modp$ an integer $\tilde{m}$
which minimizes $|q|$ among all $q\in [-p/2,p/2]$ with $m-q\in p\Z$.
The integer $\tilde{m}$ is possibly not unique if $p$ is even (for
instance $\widetilde{-1}=-1$ or $\widetilde{-1}=1$ if $p=2$),
nevertheless $|\tilde{m}|$ is uniquely determined, and
$|\widetilde{-m}|=|\tilde{m}|$.

We define \emph{reduction of $T$ $\modp$} a current obtained from
$T$ by taking the reduction of its multiplicity $\modp$, namely
$$
T^p:=\sum_{i=1}^\infty (f_i)_\sharp \segop (\tilde{\theta}\circ
f_i)\chi_{K_i}\segcl
$$
whenever $T=\sum_i (f_i)_\sharp\segop(\theta\circ
f_i)\chi_{K_i}\segcl$. Obviously any reduction $T^p$ has integer
multiplicity in $[-p/2,p/2]$ and it is equivalent to $T$ $\modp$.
The reduction is not unique, because of the ambiguity on the sign of
the multiplicity and on the choice of the reduction from $\Z$ to
$[-p/2,p/2]$, but since $|\widetilde{-m}|=|\tilde{m}|$ it turns out
that $|\tilde\theta|$ is nonzero and uniquely determined by $T$ on
$S_T$, up to $\Haus{k}$-negligible sets.

The following proposition shows that elements of $\fcp{0}{E}$ are
equivalence classes of currents in $\rc{0}{E}$ and provides a basic
lower semicontinuity property.

\begin{proposition}[Characterization of $\fcp{0}{E}$]\label{pscuola}
Let $E$ be a compact length space, let $[R]\in\fcp{0}{E}$ and let
$T_h\in\rc{0}{E}$ be such that $[T_h]\to [R]$ in $\fcp{0}{E}$ and
$\sup_h\|T_h\|(E)$ is finite. Then there exists $T\in\rc{0}{E}$ such
that $[T]=[R]$ and $\liminf\|T_h^p\|(E)\geq\|T\|(E)$.
\end{proposition}

\noindent {\sc Proof.} We assume without loss of generality that the
$\liminf$ is a finite limit and write
\[
T_h=\sum_{i=1}^{N_h}\theta_{h,i}\delta_{x(h,i)}
\]
with $\theta_{h,i}\in\Z\setminus\{0\}$. We can also assume, possibly
replacing $T_h$ by their reductions, that $\theta_{h,i}\in
[-p/2,p/2]$, so that $T_h=T_h^p$. We have $N_h\leq\sup_h\|T_h\|(E)$
and we can assume (possibly extracting once more a subsequence) that
$N_h=N$ is independent of $h$. Furthermore, we can also assume that
$x(h,i)\to x(i)$ as $h\to\infty$ and
\[
\theta_{h,i}=\theta_i\in [-p/2,p/2]\setminus\{0\}\qquad\text{for $h$
large enough}
\]
for all $i=1,\ldots,N$. Since $E$ is a length space we can find
currents $G_{h,i}\in\ic{1}{E}$ (induced by geodesics joining
$x(h,i)$ to $x_i$) with $\partial
G_{h,i}=\delta_{x(h,i)}-\delta_{x(i)}$ and ${\bf M}(G_{h,i})\to 0$,
for $i=1,\ldots,N$. Since
$T_h-\sum_i\theta_i\delta_{x_i}=\sum\partial G_{h,i}$, it turns out
that
\[
\fflat(T_h-\sum_{i=1}^N\theta_i\delta_{x(i)})\to 0,
\]
whence $[R]=[\sum_1^N\theta_i\delta_{x_i}]$ $\modp$. Also, it
follows that
\[
\|\sum_{i=1}^N\theta_i\delta_{x_i}\|(E)\leq\sum_{i=1}^N|\theta_i|\leq\liminf_{h\to\infty}
\sum_{i=1}^N|\theta_{h,i}|=\liminf_{h\to\infty}\|T_h\|(E).
\]
\qed

In the next theorem we characterize ${\bf M}_p$ on $\rc{k}{E}$.

\begin{theorem}\label{teosabato}
Let $T\in\rc{k}{E}$, with $E$ compact length space. Then ${\bf
M}_p(T)=\|T^p\|(E)$, where $T^p$ is any reduction of $T$ modulo $p$.
In particular, the additivity property holds with $\|T\|_p=\|T^p\|$.
\end{theorem}

\noindent {\sc Proof.} The inequality ${\bf M}_p(T)\leq\|T^p\|(E)$ is
obvious, because $T^p=T$ $\modp$. We shall prove the converse
inequality by induction on $k$. Without loss of generality we can
assume that $E$ is a compact convex subset of a Banach space (indeed,
an isometric embedding does not increase the ${\bf M}_p$ mass, while
leaving $\|T^p\|(E)$ unchanged). The inequality is equivalent to the
lower semicontinuity of $T\mapsto\|T^p\|(E)$ under
$\fflat_p$-convergence. More generally, we shall prove by induction on
$k$ that
\[
\|T^p\|(A)\leq\liminf_{h\to\infty}\|T^p_h\|(A)
\]
for all open sets $A\subset E$ whenever $\fflat_p(T_h-T)\to 0$.

\noindent ($k=0$). Let $T\in\rc{0}{E}$ and let $T_h\in\rc{0}{E}$ be
satisfying $\fflat_p(T_h-T)\to 0$; we fix an open set $A\subset E$
and we assume with no loss of generality that the $\liminf$ above is
a limit and that $T_h=T_h^p$. Then, we are allowed to extract
further subsequences and we can assume that the fast convergence
condition $\sum_h\fflat_p(T_h-T)<\infty$ holds. Let $u$ be the
distance function from $E\setminus A$ and apply for $\Leb{1}$-a.e.
$r>0$ Proposition~\ref{pscuola} to $T_h\res\{u>r\}$ and
$[T\res\{u>r\}]$ to obtain the existence of $S_r\in\rc{0}{E}$ with
$S_r=T\res\{u>r\}$ $\modp$ and
$$
\|S_r\|(E)\leq\liminf_{h\to\infty}\|T_h\|(\{u>r\}).
$$
Since $S_r=T^p\res\{u>r\}$ $\modp$ as well, it follows that
$$
\|T^p\|(\{u>r\})\leq\|S_r\|(E)\leq\liminf_{h\to\infty}\|T_h\|(\{u>r\})
\leq\liminf_{h\to\infty}\|T_h\|(A).
$$
Letting $r\downarrow 0$ the lower semicontinuity property on $A$
follows.

\noindent (Induction step). Let us prove that the induction
assumption gives
\[
\liminf_h\|T_h^p\res du\|(A)\geq\|T^p\res du\|(A)
\]
whenever $T_h\to T$ in $\fcp{k}{E}$.  Indeed, assuming with no loss of
generality that
\[
\sum_h\fflat_p(T_h-T)<\infty,
\]
we know from the definition of the slice operator and
\eqref{cachanbis} that
$$
\lim_{h\to\infty}\langle T_h,u,r\rangle=\langle T,u,r\rangle \qquad
\text{in $\fcp{k}{E}$}
$$
for $\Leb{1}$-a.e. $r\in\R$, hence Proposition~\ref{pslicei} gives
\begin{eqnarray*}
\liminf_{h\to\infty}\|T^p_h\res du\|(A)&=&
\liminf_{h\to\infty}\int_{{\bf R}} \|\langle
T_h^p,u,r\rangle\|(A)\,dr  \geq \int_{{\bf R}}
\liminf_{h\to\infty}\|\langle T_h^p,u,r\rangle\|(A)\,dr\\
&\geq&\int_{{\bf R}}\|\langle T^p,u,r\rangle\|(A)\,dr= \|T^p\res
du\|(A).
\end{eqnarray*}
By applying Proposition~\ref{pmasso} to $T^p\res A$ we have
$$
\|T^p\|(A)= \sup\left\{ \sum_{i=1}^N \|T^p\res
d\pi^i\|(A_i)\right\},
$$
where the supremum runs among all finite disjoint families of open
sets $A_1,\ldots,A_N\subset A$ and all $N$-ples of $1$-Lipschitz
maps $\pi^i$. By the previous step all the finite sums are lower
semicontinuous with respect to $\fflat_p$ convergence, whence the
lower semicontinuity of $T\mapsto\|T^p\|(A)$ follows.

\noindent This concludes the proof of the equality ${\bf
M}_p(T)=\|T^p\|(E)$. Since for $T\in\rc{k}{E}$ and $u\in {\rm
Lip}(E)$ it holds
$$
(T\res\{u<r\})^p=T^p\res\{u<r\} \qquad\text{for $\Leb{1}$-a.e.
$r\in\R$},
$$
it follows that the additivity property is fulfilled with
$\|T\|_p:=\|T^p\|$. \qed

\subsection{Isoperimetric inequalities $\modp$}

Having defined $\rc{k}{E}$, we define
$$
\rcp{k}{E}:=\left\{[T]:\ T\in\rc{k}{E}\right\}.
$$
An open problem, in connection with the ${\bf M}_p$, mass is the
validity of the analogous of \eqref{comedia}, namely
$$
\left\{[T]\in\fcp{k}{E}:\ {\bf M}_p([T])<\infty\right\}=\rcp{k}{E}.
$$
We plan to investigate this in \cite{ambkatzwen}.

We also define
\begin{equation}\label{icpke}
\icp{k}{E}:=\left\{[T]:\ [T]\in\rcp{k}{E},\,\,
[\partial T]\in\rcp{k}{E}\right\}.\end{equation}

\begin{theorem} \label{tcheck}
Let $E$ be a compact convex subset of a separable Banach space. Then
${\bf M}_p$ and $\icp{k}{E}$, as defined in \eqref{mp} and
\eqref{icpke} respectively, satisfy conditions (i), (ii), (iii),
(iv) of Section~\ref{swenger} with constants depending on $k$ only.
\end{theorem}

\noindent {\sc Proof.} (i) The fact that the slice operator maps
$\rcp{k}{E}$ into $\rcp{k-1}{E}$ follows by the fact the slice
preserves integer rectifiability, see Proposition~\ref{pslicei}.
Since the boundary operator and the slice commute (up to a change of
sign) the slice operator maps also $\icp{k}{E}$ into $\icp{k-1}{E}$.
In order to prove \eqref{coarea} we consider the inequality in an
integral form, namely
\begin{equation}\label{morel1}
\int_a^{*b}{\bf M}_p(\langle [T],u,r\rangle)\,dr\leq {\rm
Lip}(u)(\|T\|_p\bigl(\{u<b\})-\|T\|_p(\{u<a\})\bigr) \qquad
-\infty<a\leq b<+\infty.
\end{equation}
For $S\in\ic{k}{E}$ we can apply \cite[Theorem~5.6]{ak2} to obtain
$$
\int_a^b{\bf M}(\langle S,u,r\rangle)\,dr\leq {\rm
Lip}(u)(\|S\|\bigl(\{u<b\})-\|S\|(\{u<a\})\bigr).
$$
Now, let $(S_i)\subset\ic{k}{E}$ be such that
$\sum_i\fflat_p(S_i-T)<\infty$ and ${\bf M}(S_i)\to {\bf M}_p([T])$;
we have seen in the proof of Theorem~\ref{teomassp} that there
exists an at most countable set $N$ such that ${\bf
M}(S_i\res\{u<r\})\to \|T\|_p(\{u<r\})$ for all $r\in\R\setminus N$;
in addition, the fast convergence assumption ensures that
$\fflat_p(\langle S_i,u,r\rangle-\langle T,u,r\rangle)\to 0$ for
$\Leb{1}$-a.e. $r>0$. So, passing to the limit in the previous
inequality with $S=S_i$, Fatou's lemma and the lower semicontinuity
of ${\bf M}_p$ provide \eqref{morel1} when $a,\,b\notin N$. In the
general case the inequality can be recovered by monotone
approximation.

\noindent (ii) In the proof of this property we shall use properties
(i), (iii) and (iv) which are estabilished independently of (iii).
In the case $k=1$, property (iv) holds with the explicit constant
$A_k=2$; furthermore (iii) holds with $c^*=2$. For all
$[L]\in\icp{1}{E}$ with $\partial [L]=0$ we shall be able to
construct a family of currents $[L_i]$ with the same properties
satisfying
\begin{equation}\label{decompo1}
{\bf M}_p([L]-\sum_{i=1}^\infty[L_i])=0, \qquad {\bf M}_p([L])
=\sum_{i=0}^\infty{\bf M}_p([L_i])
\end{equation}
and ${\rm diam}({\rm supp}([L_i]))\leq 8 {\bf M}_p([L_i])$. Given
this decomposition, an application of property (iii) to all $[L_i]$
provides currents $[T_i]$ with $\partial [T_i]=[L_i]$ and ${\bf
M}_p([T_i])\leq 16\bigl[{\bf M}_p([L_i])\bigr]^2$ and we can apply
the property (iii) to find $[S_N]$ with $\partial
S_N=[L]-\sum_1^N[L_i]$ and ${\bf M}_p([S_N])\to 0$; it turns that
for $N$ large enough the current
$$[T]:=\sum_{i=1}^N[T_i]+S_N$$
has the required property.

In order to achieve the decomposition \eqref{decompo1} it suffices
to find finitely many, say $N$, currents $[L_i]$ with ${\rm
diam}({\rm supp}([L_i]))\leq 8 {\bf M}_p([L_i])$,
\begin{equation}\label{decompo2}
{\bf M}_p([L]-\sum_{i=1}^N[L_i])\leq \frac{4}{5} {\bf M}_p([L]),
\qquad {\bf M}_p([L])={\bf
M}_p([L]-\sum_{i=1}^N[L_i])+\sum_{i=1}^N{\bf M}_p([L_i])
\end{equation}
and then iterate this decomposition (first to $[L]-\sum_1^N[L_i]$
and so on) countably many times. In order to obtain the
decomposition \eqref{decompo2} we apply Lemma~3.2 of \cite{We} with
$F=1/2$ and $\mu=\|T^p\|$ (since $A_1=2>F$ this choice ensures that
for $\mu$-a.e. $x$ there exists $r>0$ such that $\mu(B_r(x))\geq
Fr$) to obtain finitely many points $y_1,\ldots,y_N$ and
corresponding radii $r_i>0$ satisfying:
\begin{itemize}
\item[(a)]  $\mu(B_{r_i}(y_i))\geq Fr_i$ and $\mu(B_s(y_i))<Fs$ for all $s>r_i$;
\item[(b)] the balls $B_{2r_i}(y_i)$ are disjoint;
\item[(c)] $5\sum_1^N\mu(B_{r_i}(y_i))\geq\mu(E)$.
\end{itemize}
Since (a) gives
$$
\int_{r_i}^{*2r_i}{\bf M}_p(\langle [L],d(\cdot,y_i),r\rangle)\,dr
\leq {\bf M}_p([L]\res B_{2r_i}(y_i))<2Fr_i=r_i
$$
we know that ${\bf M}_p(\langle [L],d(\cdot,y_i),r\rangle)<1$ in a
set of positive $\Leb{1}$-measure in $(r_i,2r_i)$. But since the
slices belong to $\icp{0}{E}$ it follows that ${\bf M}_p(\langle
[L],d(\cdot,y_i),r\rangle)=0$ in a set of positive $\Leb{1}$-measure
in $(r_i,2r_i)$. Choosing $\eta_i\in (r_i,2r_i)$ in such a way that
$\langle [L],d(\cdot,y_i),\eta_i\rangle)=0$ we can define
$$
[L_i]:=[L]\res\{d(\cdot,y_i)<\eta_i\},\qquad 1\leq i\leq N.
$$
Our choice of $\eta_i$ ensures that $\partial [L_i]=0$ and property
(b) ensures that the supports of these chains are pairwise disjoint.
Also,
$$
{\rm diam}({\rm supp}([L_i]))\leq 2\eta_i\leq 4r_i\leq
8\mu(B_{r_i}(y_i))\leq 8{\bf M}_p([L_i]).
$$
Property (c) ensures that $5\sum_1^N{\bf M}_p([L_i])\geq {\bf
M}_p([L])$, so that \eqref{decompo2} holds.

\noindent (iii) This can be easily achieved by a cone construction
as, for instance, in \cite[Proposition~10.2]{ak2}. This construction
provides the constant $c^*=2$.

\noindent (iv) If $T\in\rc{k}{E}$ and $T^p$ is a reduction $\modp$,
since its multiplicity is at least $1$ we know by
\cite[Theorem~9.5]{ak2} that
$$
\|T^p\|\geq\lambda\Haus{k}\res S,
$$
where $S=S(T^p)$ is defined in \eqref{ST} with $T^p$ in place of $T$
and $\lambda$ is an ``area factor'' depending only on $S$. In
addition, \cite[Lemma~9.2]{ak2} provides the universal lower bound
$\lambda\geq k^{-k/2}$. Finally, taking into account (see
\cite{kir1}) that any countably $\Haus{k}$-rectifiable set with
finite $\Haus{k}$-measure $S$ satisfies
$$
\liminf_{r\downarrow 0}\frac{\Haus{k}(S\cap B_r(x))}{\omega_kr^k}=
1\qquad\text{for $\Haus{k}$-a.e. $x\in S$,}
$$
with $\omega_k$ equal to the Lebesgue measure of the unit ball in
$\R^k$, we obtain that (iv) holds with $A_k=k^{-k/2}\omega_k$. \qed

As a consequence, we can obtain isoperimetric inequalities in the
case when the cycle belongs to $\icp{k}{E}$ (resp. $\fcp{k}{E}$) and
the filling belongs to $\ic{k+1}{E}$ (resp. $\fcp{k+1}{E}$). In this
connection, notice that in the class of integer multiplicity
currents we have that $L\in\fc{k}{E}$ with finite mass and $\partial
L=0$ implies $L\in\ic{k}{E}$: indeed, writing $L=A+\partial B$ with
$A\in\rc{k}{E}$ and $B\in\rc{k+1}{E}$, we have $\partial A=0$ and so
$A=\partial R$ for some $R\in\ic{k+1}{E}$. Since $L=\partial (R+B)$
the boundary rectifiability theorem gives that $L\in\ic{k}{E}$. We
plan to investigate the boundary rectifiability theorem and further
properties of currents $\modp$ in \cite{ambkatzwen}.

\begin{corollary}[Isoperimetric inequality $\modp$ in $\icp{k}{E}$ and $\fcp{k}{E}$]
\label{tisop1} Let $E$ be a compact convex subset of a separable
Banach space. For $k\geq 1$ there exist constants $\delta_k$ such
that, if $[L]\in\icp{k}{E}$ is a non zero current with bounded
support and $\partial [L]=0$ then
$$
\inf\left\{\frac{{\bf M}_p([T])}{\bigl[{\bf
M}_p([L])\bigr]^{(k+1)/k}}:\
[T]\in\icp{k+1}{E},\,\,\partial[T]=[L]\right\}\leq\delta_k.
$$
The same property holds when $[L]\in\fcp{k}{E}$, taking the infimum
among all $[T]\in\fcp{k+1}{E}$ with $\partial [T]=[L]$.
\end{corollary}

\noindent {\sc Proof.} If $[L]\in\icp{k}{E}$, we want to apply
Theorem~\ref{tisop}. To this aim, it suffices to combine
Theorem~\ref{tcheck} and Theorem~\ref{teosabato}. In the general
case $[L]\in\fcp{k}{E}$, let $P_i\in\ic{k}{E}$ be satisfying
$\fflat_p(P_i-L)\to 0$ and ${\bf M}(P_i)\to {\bf M}_p(L)$. Let us
write
$$
P_i=L+A_i+\partial B_i+pQ_i
$$
with $A_i\in\rc{k}{E}$, $B_i\in\rc{k+1}{E}$, $Q_i\in\fc{k}{E}$ and
${\bf M}(A_i)+{\bf M}(B_i)\to 0$. We have $[\partial P_i]=[\partial
A_i]$, and since $[P_i-A_i]\in\icp{k}{E}$ we can find currents
$[P_i']\in\icp{k+1}{E}$ with $\partial [P_i']=[P_i-A_i]$ and
$$
{\bf M}_p([P_i'])\leq\delta_k\left[{\bf
M}_p([P_i-A_i])\right]^{(k+1)/k}\leq \delta_k\left[{\bf
M}_p([L])\right]^{(k+1)/k}+\omega_i
$$
with $\omega_i$ infinitesimal. It is now immediate to check that
$\partial [P_i'-B_i]=[L]$, so that $[P_i'-B_i]\in\fcp{k+1}{E}$, and
that
$$
\limsup_{i\to\infty}{\bf M}_p([P_i'-B_i])\leq\delta_k\left[{\bf
M}_p([L])\right]^{(k+1)/k}.
$$
\qed

\section{Filling radius inequality}\label{skatz}

In this section we investigate the validity of a filling radius
inequality, which complements the isoperimetric inequality of
Corollary~\ref{tisop1}. To this aim, for $[L]\in\icp{k}{E}$ with
$\partial [L]=0$ we consider the subspace ${\mathcal M}$ defined by
\begin{equation}\label{calmc}
{\mathcal M}:=\left\{[T]\in\fcp{k+1}{E}:\ \partial [T]=[L],\,\,{\bf
M}_p([T])<\infty\right\}.
\end{equation}
By Corollary~\ref{tisop1} ${\mathcal M}$ contains $[\bar
T]\in\icp{k+1}{E}$ with ${\bf M}_p([\bar T])\leq\delta_k \bigl[{\bf
M}_p([L])\bigr]^{(k+1)/k}$.

\begin{theorem}
\label{tkatz1}
Assume that $E$ is a compact convex subset of a separable Banach
space.
 Let $[L]\in\icp{k}{E}$ with
${\bf M}_p([L])<\infty$ and $\partial [L]=0$. Then, the infimum of
the numbers $r$ such that there exists $[T]\in\icp{k+1}{E}$
satisfying $\partial [T]=[L]$ whose support is contained in the
$r$-neighbourhood of ${\rm supp\,}[L]$ is not greater than
$C_k\bigl[{\bf M}_p([L])\bigr]^{1/k}.$\\
The constant $C_k$ depends only on $k$ and on the constant
$\delta_k$ in Corollary~\ref{tisop1}.
\end{theorem}

\noindent {\sc Proof.} We claim that the infimum is unchanged if we
look for fillings in the more general class $\fcp{k+1}{E}$. Indeed,
let $[S]\in\fcp{k+1}{E}$ with $\partial [S]=[L]$ whose support is
contained in the $r$-neighbourhood of $K$, and let $u$ be the
distance function from $K$, the support of $[L]$. We consider a
sequence $(S_i)\subset\ic{k+1}{E}$ with
$\sum_i\fflat_p(S_i-S)<\infty$ and $r'>r$. We know that for
$\Leb{1}$-a.e. $\rho\in (r,r')$ we still have
$[S_i\res\{u<\rho\}]\to [S\res\{u<\rho\}]$ in $\fcp{k+1}{E}$, and
since $[S\res\{u<\rho\}]=[S]\res\{u<\rho\}=[S]$ we see that,
possibly replacing $S_i$ by $S_i\res\{u<\rho\}$, there is no loss of
generality in assuming that the supports of $S_i$ are contained in
the $\rho$-neighbourhood of $K$, for some $\rho<r'$. Now, let us fix
$i$ and write
$$
S-S_i=A+\partial B+pQ
$$
with $A\in\rc{k+1}{E}$, $B\in\rc{k+2}{E}$, $Q\in\fc{k+1}{E}$. For
$\Leb{1}$-a.e. $t\in (\rho,r')$ we can restrict both sides to
$\{u<t\}$ to obtain
$$
S-S_i=A\res\{u<t\}-\langle B,u,t\rangle+\partial
(B\res\{u<t\})+pQ\res\{u<t\}.
$$
It follows that the current $[S_i+A\res\{u<t\}-\langle
B,u,t\rangle]\in\icp{k+1}{E}$ has boundary $[L]$ and support
contained in the $r'$-neighbourhood of $K$. Since $r'>r$ is
arbitrary, this proves the claim.

So, from now on we look for $[S]\in\fcp{k+1}{E}$ with
$\partial[S]=[L]$ and we set
\[
c:=\delta_k [{\bf M}_p([L])]^{(k+1)/k}.
\]

\section{Ekeland principle}\label{sekeland}

Let us recall the Ekeland variational principle \cite{ET} (see also
the proof in \cite{ET1}, using only the countable axiom of choice):
If $(X,d)$ is a complete metric space and $f:X\to\R\cup\{+\infty\}$
is lower semicontinuous and bounded from below, then for all
$\eps>0$ there exists $y\in X$ such that $x\mapsto f(x)+\eps d(x,y)$
attains its minimum value at $x=y$.  Since ${\bf M}_p\geq\fflat_p$
and is $\fflat_p$ lower semicontinous, we know that ${\mathcal M}$
is a complete metric space, when endowed with the distance induced
by ${\bf M}_p$. Let $\eps>0$ be fixed; the lower semicontinuity of
$[T]\mapsto{\bf M}_p([T])$ ensures that we can apply the Ekeland
variational principle to find $[S]\in {\mathcal M}$ such that
$$
[T]\mapsto{\bf M}_p([T])+\eps{\bf M}_p([T]-[S]) \qquad\qquad
[T]\in{\mathcal M}
$$
is minimal at $[T]=[S]$. If $\eps\leq 1/2$, the minimality of $[S]$ gives
\begin{equation}\label{dgcl1}
{\bf M}_p([S])\leq {1+\eps\over 1-\eps}{\bf M}_p([\bar T])\leq 3c.
\end{equation}
Let us now prove the density lower bound
\begin{equation}\label{dgcl}
\|S\|_p(B_\ro(x))\geq {(3\delta_k)^{-k}\over (k+1)^{k+1}}\ro^{k+1}
\qquad\text{for all $\ro\in (0,\tau(x))$}
\end{equation}
for any $x\in {\rm supp\,}[S]\setminus K$; here $\tau(x)={\rm
dist}(x,K)>0$.
 In order to prove (\ref{dgcl}) we use a standard
comparison argument based on the isoperimetric inequalities: let
$x\in{\rm supp\,}[S]\setminus K$: for $\Leb{1}$-a.e. $\ro>0$ the
slice
$$[S_\ro]:=\langle [S],d(\cdot,x),\ro\rangle=
\partial ([S]\res \{d(\cdot,x)<\ro\})-
(\partial [S])\res\{d(\cdot,x)<\ro\}$$ belongs to $\fcp{k}{E}$ and
has no boundary, because the conditions $\rho<\tau(x)$ and
$\partial[S]=[L]$ imply
\[
(\partial [S])\res\{d(\cdot,x)<\ro\}=0.
\]
By Corollary~\ref{tisop1} we can find $[R]\in\fcp{k+1}{E}$ with
$\partial[R]=[S_\ro]$ and
\begin{equation}
\label{monge1}
{\bf M}_p([R])\leq\delta_k\bigl[{\bf M}_p([S_\ro])\bigr]^{(k+1)/k}.
\end{equation}
Comparing $[S]$ with
\[
[S']:=[S]\res \left(E\setminus B_\ro(x)\right)+[R]
\]
we find
\begin{eqnarray*}
{\bf M}_p([S])&\leq&
{\bf M}_p([S'])+\eps\fflat_p([S]\res B_\ro(x)-R)\leq{\bf M}_p([R])
+{\bf M}_p([S]\res (E\setminus B_\ro(x)))\\
&+&\eps {\bf M}_p([S]\res B_\ro(x))+\eps {\bf M}_p([R]),
\end{eqnarray*}
so that
\begin{equation}
\label{monge2}
{\bf M}_p([S]\res B_\ro(x))\leq\frac{1+\eps}{1-\eps}{\bf M}_p([R])\leq
3{\bf M}_p([R]).
\end{equation}
By \eqref{monge1} and \eqref{monge2} it follows that
\[
\|S\|_p(B_\ro(x))\leq
3\delta_k\left[\frac{d}{d\ro}\|S\|_p(B_\ro(x))\right]^{(k+1)/k}
\]
for $\Leb{1}$-a.e. $\ro>0$. Since $\|S\|_p(B_\ro(x))>0$ for any
$\ro>0$ (because $x\in {\rm supp\,}[S]$), this proves that
\[
\ro\mapsto\bigl(\|S\|_p(B_\ro(x))^{1/(k+1)}-(3\delta_k)^{-k/(k+1)}\ro/(k+1)
\]
nondecreasing, and hence nonnegative, in $(0,\tau(x))$.

To obtain that the estimate on the support of $[S]$ it suffices to
take a sequence $\ro_i\uparrow\tau(x)$ and to use the inequalities
$$
\|T\|_p(B_\ro(x))\leq {\bf M}_p([S])\leq 3c\leq 3\delta_k {\bf
M}_p([L])]^{(k+1)/k}
$$ to obtain that $\tau(x)$ can be bounded by a multiplicative
constant times ${\bf M}_p([L])]^{1/k}$.  Since~$x$ is arbitrary this
proves that the support of $[S]$ is contained in the $r$-neighbourhood
of~$K$, with $r\leq C_k\bigl[{\bf M}_p([L])\bigr]^{1/k}$. \qed

\begin{remark} [Extension to $\fcp{k}{E}$]
{\rm The same property holds, with the same proof, in the classes
$\fcp{k}{E}$, namely: for all $[L]\in\fcp{k}{E}$ with ${\bf
M}_p([L])<\infty$ and $\partial [L]=0$ the infimum of the numbers $r$
such that there exists $[T]\in\fcp{k+1}{E}$ satisfying $\partial
[T]=[L]$ whose support is contained in the $r$-neighbourhood of ${\rm
supp\,}[L]$ is not greater than~$C_k\bigl[{\bf M}_p([L])\bigr]^{1/k}.$
}
\end{remark}

\section{Nonorientable manifolds and currents $\modtwo$}

Let $(M,g)$ be a compact $n$-dimensional Riemannian manifold without
boundary and let $\tau$ be a Borel orientation of $M$, i.e. a Borel
choice of unit vectors $\tau_1,\ldots,\tau_n$ spanning the tangent
space and mutually orthogonal (the construction can be easily
achieved in local coordinates and gluing, by the minimal Borel
regularity imposed on $\tau$, is not a problem), possibly up to
$\Haus{n}$-negligible sets. Here $\Haus{n}$ is the Hausdorff
$n$-dimensional measure induced by the Riemannian distance. Of
course, when $M$ is not orientable any orientation $\tau$ is
necessarily discontinuous and it is by no means canonical. In any
case, given this orientation, we can define a current $\segop
M\segcl\in\rc{n}{M}$ as follows:
$$
\segop M\segcl(f d\pi_1\wedge\ldots\wedge d\pi_k):=\int_M f{\rm
det}\bigl(\frac{\partial\pi_i}{\partial\tau_j}\bigr) \,d\Haus{n}.
$$
While $\segop M\segcl$ is not canonical, its equivalence class
$\modtwo$ obviously is, because different orientations induce
currents $\segop M\segcl$ equivalent $\modtwo$. In connection with
mass measures, it is not difficult to check that
$$
\|\segop M\segcl\|(B)=\Haus{n}(B)\qquad\text{for all $B\subset M$
Borel}
$$
(or, it suffices to apply Lemma~9.2 and Theorem~9.5 of \cite{ak2},
valid in a much more general context). In turn, $\Haus{n}$ coincides
with the Riemannian volume measure, see for instance
\cite[3.2.46]{federer}. Passing to the equivalence class the same is
true, because $\segop M\segcl$ is already reduced $\modtwo$, hence
$\|\segop M\segcl\|_2=\|\segop M\segcl\|$ and their total mass is
${\rm Vol}(M)$.

We are now going to show that $\partial \segop M\segcl=0$ $\modtwo$,
and we prove this fact building a ``nice'' current on $M$ as the
image of the exponential map ${\rm Exp}_O$ at some base point $O\in
M$. As the referee pointed out, for the purpose of proving $\partial
\segop M\segcl=0$ $\modtwo$ simpler proofs are possible, which apply
to Lipschitz manifolds as well; on the other hand, we believe that
this global construction (which uses some properties of the cut
locus estabilished only recently) might have an independent
interest.

\begin{theorem}\label{ttanaka}
Let $(M,g)$ be a compact $n$-dimensional Riemannian manifold with no
boundary. Then $\partial \class =0$ and, in particular, $\class \in
{\bf I}_{2,n}(M)$.
\end{theorem}

\noindent {\sc Proof.} We fix a base point $O\in M$ and consider the
distance function $u$ from $O$. We consider the tangent cut locus
$TC$ at $O$, namely $v\in T_OM$ belongs to $TC$ if and only if
$\exp_O(tv)$ is the unique minimizing geodesic in $[0,\tau]$ for all
$\tau<1$, and it is nonminimizing in $[0,\tau]$
for all $\tau>1$. It turns out that $TC$ is locally a Lipschitz
graph \cite{tanaka,nirenberg}, and that the boundary of the
star-shaped region
$$
\Omega:=\left\{tv:\ v\in TC,\,\,t\in [0,1]\right\}
$$
is contained in $TC$. Of course the exponential map ${\rm Exp}_O$
maps $TC$ into the cut locus, that we shall denote by $C$.

Next, we consider some additional regularity properties of $u$,
besides $1$-Lipschitz continuity: this function is locally
semiconcave out of $O$, namely in local coordinates its
second derivatives are locally bounded from above in $M\setminus\{O\}$.
This implies, by standard results about semiconcave functions and viscosity solutions to the
Hamilton-Jacobi equation $g_x(\nabla u,\nabla u)=1$ the following
facts (for (i), (ii), (iii) see for instance \cite{mantegazza}; for
the more delicate property (iv) see \cite[Theorem~4.12]{mennucci} or
the appendix of \cite{figalli}):

\begin{itemize}
\item[(i)] for all $x\neq O$ the set of supergradients
$$
\partial^+u(x):=\left\{v\in T_xM:\ u({\rm exp}_x(w))\leq
u(x)+g_x(v,w)+o(|w|)\right\}
$$
is convex and not empty, and $u$ is differentiable at $x$ if and
only if $\partial^+u(x)$ is a singleton;
\item[(ii)] for all $x\neq O$
 the closed convex hull of $\partial^+u(x)\cap\{v\in T_xM:\ g_x(v,v)=1\}$
coincides with $\partial^+u(x)$ and the former set is in 1-1
correspondence with final speeds of minimizing unit speed geodesics
joining $O$ to $x$;
\item[(iii)] for $j$ integer the set $\left\{x\in M:\ {\rm
dim}\bigl(\partial^+u(x)\bigr)\geq j\right\}$ has $\sigma$-finite
$\Haus{n-j}$-measure;
\item[(iv)] the set of points $x\in C$ where $u$ is differentiable
is $\Haus{n-1}$-negligible.
\end{itemize}

Now, we fix an orientation of $T_OM$ and we consider the canonical
(Euclidean) $n$-current $\segop\Omega\segcl\in\rc{n}{T_OM}$, with
multiplicity 1 on $\Omega$ and $0$ on $T_OM\setminus\Omega$ induced
by this orientation; since
\[
\Haus{n-1}(\partial\Omega)\leq\Haus{n-1}(TC)<\infty
\]
we know that $\segop\Omega\segcl\in\ic{n}{T_OM}$ and its boundary is
supported on $TC$. Then, we consider its image $T=({\rm
exp}_O)_\sharp\segop\Omega\segcl\in\ic{n}{M}$ via the exponential
map. We are going to prove that:
\begin{itemize}
\item[(a)] $T=\segop M\segcl$ for some orientation of $M$;
\item[(b)] $\partial T=2R$ for some $R\in\ic{n-1}{M}$.
\end{itemize}
These two facts imply the stated properties of $\segop M\segcl$. In
connection with (a), notice first that ${\rm exp}_O(\Omega)=M$,
because for each point $x\in M$ there is at least one minimizing
geodesic to $O$, and it is unique before reaching $x$. Moreover,
Rademacher's theorem implies that $\Haus{n}$-a.e. point $x\in M$ is
a differentiability point of $u$, so that $\partial u^+(x)=\{\nabla
u(x)\}$ is a singleton and there is a unique minimizing constant
speed geodesic between $O$ and $x$ (since its final speed is
uniquely determined, ODE uniqueness applies); if $v$ is the initial
speed of this geodesic, it turns out that $x={\rm exp}_O(d(O,x)v)$
and $t d(O,x) v\in\Omega$ for all $t<1$, hence $d(O,x)v\in\Omega$.
This proves that ${\rm exp}_O$ has a unique inverse $\Haus{n}$-a.e.;
these facts imply that $T=\segop M\segcl$ provided we choose as
orientation of $M$ the one induced by $T_OM$ via the exponential map
${\rm exp}_O$.

In connection with (b), we know that $\partial T=({\rm
exp}_O)_\sharp(\partial \segop\Omega\segcl)$ and that
$\partial\segop\Omega\segcl$ is a current with unit multiplicity
$\Haus{n-1}$-a.e. on $\partial\Omega$, because $TC$ is locally a
Lipschitz graph. We claim that for $\Haus{n-1}$-a.e. $x\in C$ the
pre-image ${\rm exp}_O^{-1}(x)$ contains exactly two points. Since
the multiplicity of $\partial T$ at $x$ can be obtained adding the
properly multiplicities of $\partial \segop\Omega\segcl$ at ${\rm
exp}_O^{-1}(x)$, this proves that $\partial T$ has an even
multiplicity. To prove the claim, we know by (iv) that for
$\Haus{n-1}$-a.e. $x\in C$ the number of minimizing geodesics is
strictly greater than 1; on the other hand, (iii) with $j=2$ gives
that for $\Haus{n-1}$-a.e. $x\in C$ the dimension of
$\partial^+u(x)$ is at most 1, hence the extreme points are at most
two: therefore there exist precisely two minimizing geodesics from
$O$ to $x$ at $\Haus{n-1}$-a.e. $x\in C$.
 \qed

\noindent {\bf Proof of Theorem~\ref{tmain}.} It suffices to apply
Theorem~\ref{tkatz1} with $k=n$. To this aim, we consider the
canonical current $\class$ associated to $M$. By
Theorem~\ref{ttanaka} this current belongs to ${\bf I}_{2,n}(M)$ and
it is a cycle $\modtwo$. Then, given an isometric embedding $i$ of
$M$ into a (separable) Banach space $F$, we consider the closed
convex hull $E$ of $i(M)$ (which is a compact set, by the
compactness of $i(M)$), and apply Theorem~\ref{tkatz1} to the cycle
$[L]=i_\sharp\class\in{\bf I}_{2,n}(E)$, whose ${\bf
M}_2$ mass is (by the isometric invariance of the
${\bf M}_2$-mass of rectifiable currents) equal to
${\bf M}_2(\segop M\segcl)={\rm Vol}(M)$.

\section{Appendix}

In this appendix we recall the basic definitions of the metric
theory developed in \cite{ak2}.

\begin{definition} Let $k\geq 1$ be an integer.
We denote by ${\mathcal D}^k(E)$ the set of all $(k+1)$-ples
$\omega=(f,\pi_1,\ldots,\pi_k)$ of Lipschitz real valued functions
in $E$ with the first function $f$ in ${\rm Lip}_b(E)$. In the case
$k=0$ we set ${\mathcal D}^0(E)={\rm Lip}_b(E)$.
\end{definition}

\begin{definition}[Metric functionals]
We call $k$-dimensional metric current any function $T:{\mathcal
D}^k(E)\to\rn{}$ satisfying the following three axioms:
\begin{itemize}
\item[(a)] $T$ is multilinear;
\item[(b)] $T(f,\pi_1^n,\ldots,\pi_k^n)\to T(f,\pi_1,\ldots,\pi_k)$
whenever $\pi_i^n\to\pi_i$ pointwise and $\sup_n{\rm Lip}(\pi_i^n)$
is finite, for $1\leq i\leq k$;
\item[(c)] $T(f,\pi_1,\ldots,\pi_k)=0$ if, for some $i\in\{1,\ldots,k\}$, $\pi_i$ is
constant in a neighbourhood of the support of $f$.
\end{itemize}
We denote by $MF_k(E)$ the vector space of $k$-dimensional metric
currents.
\end{definition}

A consequence of these axioms is that $T$ is alternating in
$(\pi_1,\ldots,\pi_k)$, so the differential forms notation $f
d\pi_1\wedge\ldots\wedge d\pi_k$ can be used. We can now define an
``exterior differential''
$$
d\omega=d (f d\pi_1\wedge\ldots\wedge d\pi_k):=df\wedge
d\pi_1\wedge\ldots\wedge\pi_k
$$
mapping ${\mathcal D}^k(E)$ into ${\mathcal D}^{k+1}(E)$ and, for
$\varphi\in{\rm Lip}(E,F)$, a pull back operator
$$
\varphi^\sharp\omega=\varphi^\sharp(f d\pi_1\wedge\ldots\wedge
d\pi_k= f\circ\varphi d\pi_1\circ\varphi\wedge\ldots\wedge
d\pi_k\circ\varphi
$$
mapping ${\mathcal D}^k(F)$ on ${\mathcal D}^k(E)$. These operations induce
in a natural way a boundary operator and a push forward map for
metric functionals.

\begin{definition}[Boundary] Let $k\geq 1$ be an integer and let
$T\in MF_k(E)$. The boundary of $T$, denoted by $\partial T$, is the
$(k-1)$-dimensional metric current in $E$ defined by $\partial
T(\omega)=T(d\omega)$ for any $\omega\in {\mathcal D}^{k-1}(E)$.
\end{definition}

\begin{definition}[Push-forward] Let $\varphi:E\to F$ be a
Lipschitz map and let $T\in MF_k(E)$. Then, we can define a
$k$-dimensional metric current in $F$, denoted by $\varphi_\sharp
T$, setting $\varphi_\sharp T(\omega)=T(\varphi^\sharp \omega)$ for
any $\omega\in {\mathcal D}^k(F)$.
\end{definition}

We notice that, by construction, $\varphi_\sharp$ commutes with the
boundary operator, i.e.
\begin{equation}\label{commuta1}
\varphi_\sharp(\partial T)=\partial(\varphi_\sharp T).
\end{equation}

\begin{definition}[Restriction]
Let $T\in MF_k(E)$ and let $\omega=g d\tau_1\wedge\ldots\wedge
d\tau_m\in{\mathcal D}^m(E)$, with $m\leq k$ ($\omega=g$ if $m=0$). We
define a $(k-m)$-dimensional metric current in $E$, denoted by
$T\res\omega$, setting
$$
T\res\omega (f d\pi_1\wedge\ldots\wedge d\pi_{k-m}):= T(fg
d\tau_1\wedge\ldots\wedge d\tau_m\wedge d\pi_1\wedge\ldots\wedge
d\pi_{k-m}).
$$
\end{definition}

\begin{definition}[Currents with finite mass]\label{dmass}
Let $T\in MF_k(E)$; we say that $T$ has finite mass if there exists
a finite Borel measure $\mu$ in $E$ satisfying
\begin{equation}\label{dmass1}
|T(f d\pi_1\wedge\ldots\wedge d\pi_k)|\leq \prod_{i=1}^k{\rm
Lip}(\pi_i)\int_E |f|\,d\mu
\end{equation}
for any $fd\pi_1\wedge\ldots\wedge d\pi_k\in {\mathcal D}^k(E)$, with
the convention $\prod_i {\rm Lip}(\pi_i)=1$ if $k=0$.
\end{definition}

It can be shown that there is a minimal measure $\mu$ satisfying
(\ref{dmass1}), which will be denoted by $\Vert T\Vert$ (indeed one
checks, using the subadditivity of $T$ with respect to the first
variable, that if $\{\mu_i\}_{i\in I}\subset{\mathcal M}(E)$ satisfy
(\ref{stimamu}) also their infimum satisfies the same condition). We
call mass of $T$ the total mass of $\|T\|$, namely ${\bf
M}(T)=\|T\|(E)$.

By the density of ${\rm Lip}_b(E)$ in $L^1(E,\Vert T\Vert)$, which
contains the class of bounded Borel functions, any $T\in MF_k(E)$ with
finite mass can be uniquely extended to forms~$f\,d\pi$ with~$f$
bounded Borel, in such a way that
\begin{equation} \label{stimamu}
|T(f d\pi_1\wedge\ldots\wedge d\pi_k)|\leq\prod_{i=1}^k{\rm
Lip}(\pi_i) \int_E |f|\,d\Vert T\Vert
\end{equation}
for any $f$ bounded Borel, $\pi_1,\ldots,\pi_k\in{\rm Lip}(E)$.
Since this extension is unique we do not introduce a distinguished
notation for it.

Functionals with finite mass are well behaved under the push-forward
map: in fact, if~$T\in MF_k(E)$ the functional $\varphi_\sharp T$ has
finite mass, satisfying
\begin{equation}\label{pusho}
\Vert\varphi_{\sharp}T\Vert\leq [{\rm Lip}(\varphi)]^k
\varphi_{\sharp}\Vert T\Vert\enspace .
\end{equation}
If either $\varphi$ is an isometry or $k=0$ it is easy to check,
using (\ref{massup}) below, that equality holds in (\ref{pusho}). It
is also easy to check that the identity
$$
\varphi_\sharp T(f d\pi_1\wedge\ldots\wedge d\pi_k)= T(f\circ\varphi
d\pi_1\circ\varphi\wedge\ldots\wedge d\pi_k\circ\varphi)
$$
remains true if $f$ is bounded Borel and $\pi_i\in{\rm Lip}(E)$.

Functionals with finite mass are also well behaved with respect to
the restriction operator: in fact, the definition of mass easily
implies
\begin{equation}\label{massrest}
\Vert T\res \omega\Vert\leq \sup |g|\prod_{i=1}^m{\rm Lip}(\tau_i)
\Vert T\Vert\qquad\hbox{\rm with}\qquad \omega=g
d\tau_1\wedge\ldots\wedge d\tau_m.
\end{equation}
For metric functionals with finite mass, the restriction operator
$T\res\omega$ can be defined even though
$\omega=(g,\tau_1,\ldots,\tau_m)$ with $g$ bounded Borel, and still
(\ref{massrest}) holds; the restriction will be denoted by $T\res A$
in the special case $m=0$ and $g=\chi_A$.

Finally, we will use the following approximation results.

\begin{proposition}\label{plindenstrauss}
Let $E$ be a closed convex set of a Banach space. Then $\ic{k}{E}$
is dense in $\rc{k}{E}$ in mass norm. As a consequence $\ic{k}{E}$
is dense in $\fc{k}{E}$ in flat norm. The same holds in metric
spaces $F$ that are Lipschitz retracts of $E$.
\end{proposition}

\noindent {\sc Proof.} We argue as in \cite[Theorem~4.5]{ak2},
reducing ourselves to the approximation of currents $T\in\rc{k}{E}$
of the form $f_\sharp \segop\theta\segcl$ with $\theta\in
L^1(\R^k,\Z)$, $B\subset\R^k$ Borel, $f:B\to E$ Lipschitz and
$\theta=0$ $\Leb{k}$-a.e. out of $B$. Since $E$ is closed and
convex, the construction of \cite{jolisc} provides a Lipschitz
extension of $f$ to the whole of $\R^k$, still with values in $E$.
For $\eps>0$ given, we can choose $\theta'\in BV(\R^k;\Z)$ such that
$\int_{\R^k}|\theta-\theta'|\,dx<\eps $ to obtain that the current
$\tilde{T}:=f_{\sharp}\segop \theta'\segcl\in\ic{k}{E}$ satisfies
${\bf M}(T-\tilde{T})<\eps\bigl[{\rm Lip}(f)\bigr]^k$.

If $T\in\rc{k}{F}$ and $i:E\to F$ is a Lipschitz retraction, then we
can find a sequence $(T_n)\subset\ic{k}{E}$ converging in mass to
$T$. Then, the sequence $(i_\sharp T_n)\subset\ic{k}{F}$ provides
the desired approximation. \qed

\begin{proposition}[Characterization of mass]\label{pmasso}
Let $T\in MF_k(E)$ with finite mass. Then $\Vert T\Vert(E)$ is
representable by
\begin{equation}\label{massup}
\sup\left\{ \sum_{i=1}^N \|T\res d\pi^i\|(A_i)\right\},
\end{equation}
where the supremum runs among all finite disjoint families of open
sets $A_1,\ldots,A_N$ and all $N$-ples of $1$-Lipschitz maps
$\pi^i$.
\end{proposition}
\noindent {\sc Proof.} In \cite[Proposition~2.7]{ak2} it is proved
that
$$
\|T\|(E)= \sup\left\{ \sum_{i=1}^N \|T\res
d\pi^i_1\wedge\ldots\wedge d\pi^i_k\|(A_i)\right\},
$$
where the supremum runs among all finite disjoint families $(A_i)$
of Borel sets and $1$-Lipschitz maps $\pi^i_j$, $1\leq i\leq N$ and
$1\leq j\leq k$. Approximating Borel sets from inside with compact
sets, and then compacts sets from outside with open sets, one can
see that the supremum is the same if $(A_i)$ runs among all finite
disjoint families of open sets. By the inequalities
$$
\|T\res dq_1\wedge\ldots\wedge dq_k\|\leq\|T\res dq_1\|\leq\|T\|
$$
with $q_j$ 1-Lipschitz we obtain \eqref{massup}. \qed

\section*{Acknowledgments}

We are grateful to Stefan Wenger for many comments that helped
improve an earlier version of the manuscript, and the referee for
his detailed report.

\end{document}